\documentclass[12pt]{article}

\usepackage{fullpage,amsfonts,amsmath,amsthm,amssymb,mathrsfs,epstopdf,soul}
\usepackage[top=2.54cm, bottom=2.54cm, left=2.54 cm, right=2.54 cm]{geometry}
\usepackage{graphicx}
\usepackage{caption}
\usepackage{subcaption}
\usepackage[shortlabels]{enumitem}

\newcommand{\lab}[1]{\label{#1}}                 

 \usepackage[usenames,dvipsnames]{color}
\usepackage[shortlabels]{enumitem}

\usepackage{latexsym}

\newtheorem{thm}{Theorem}
\newtheorem{cor}[thm]{Corollary}
\newtheorem{con}[thm]{Conjecture}
\newtheorem{lemma}[thm]{Lemma}

\newtheorem{claim}[thm]{Claim}

\newtheorem{obs}[thm]{Observation}

\newtheorem{question}[thm]{Open problem}

\def\G{\mathcal{G}}
\def\R{\mathcal{R}}
\def\B{\mathcal{B}}
\def\E{\mathcal{E}}
\def\eps{\epsilon}

\def\pr{\mathbb{P}}
\def\kvec{{\boldsymbol{k}}}
\def\wvec{{\boldsymbol{w}}}
\def\trilink{{trilink}}
\def\ex{{\mathbb E}}

\newcommand{\jc}[1]{{\color{blue}[Jane: #1]}}

\newcommand{\remove}[1]{}
\newcommand{\norm}[1]{\|#1\|}
\newcommand{\eqn}[1]{{(\ref{#1})}}
\newcommand{\ind}[1]{{\boldsymbol 1}_{\{#1\}}}

\title{The dimension of sparse and co-sparse random graph orders}
\author{Pu Gao\\ University of Waterloo\\ pu.gao@uwaterloo.ca \and Arnav Kumar \\ University of Waterloo \\ a8kumar@uwaterloo.ca }
\date{}
\begin{document}
\maketitle
\begin{abstract}
A random graph order is a partial order obtained from a random graph on $[n]$ by taking the transitive closure of the adjacency relation. The dimension of the random graph orders from random bipartite graphs $\B(n,n,p)$ and from $\G(n,p)$ were previously studied when $p=\Omega(\log n/n)$ and when $p$ is not too close to 1. There is a conjectured phase transition in the sparse range at $p=1/n$. In this paper, we investigate this conjectured phase transition and  estimate the dimension of the partial orders arising from $\B(n,n,p)$ and $\G(n,p)$ when $p=O(1/n)$. For the random bipartite order, we additionally estimate its dimension in the co-sparse regime, thereby closing all previously open ranges of $p$. Finally, we establish a general upper bound on the dimension of partial orders based on their decompositions into suborders, a result that is of independent interest.
\end{abstract}

\section{Introduction}

Partially ordered sets are a classical subject of study in combinatorics and algebra, offering foundational tools and concepts that extend further to scheduling, data mining, optimization and other sciences. The study of random partial orders dates back to Kleitman and Rothschild~\cite{kleitman1970number,kleitman1975asymptotic} around 1970 who estimated the number of partially ordered sets on $n$ elements. Their enumeration result immediately implies a number of properties for a uniformly random partially ordered set (poset) on $n$ element. For instance, asymptotically almost surely (a.a.s.) such a random poset has exactly three layers, and the number of comparable pairs is approximately  $\frac{3}{16}n^2$. 
A more general model of random poset restricted to a fixed proportion $\rho$ of comparable pairs was proposed and studied (for a certain range of $\rho$) by  Deepak Dhar~\cite{dhar1978entropy}, and this model is motivated by a statistical physics model of lattice gas with long-range three-body interactions. This model was further studied by Kleitman and Rothschild~\cite{kleitman1979phase} and by Pr\"{o}mel, Steger and Taraz~\cite{promel2001phase} for different ranges of $\rho$. Winkler~\cite{winkler1985random,winkler1985connectedness} proposed a different model of random poset that comes from the intersection of $k$ uniformly random permutations on $[n]$, and Winkler proved that for fixed $k$, the resulting random poset is a.a.s.\ connected and has diameter equal to three. 
Finally, random poset arising from random graphs were introduced and studied in~\cite{barak1984maximal} and~\cite{erdos1991dimension}.  In this model we consider a random graph $\G_n$ on vertex set $[n]$. The partial order $P=P([n],<_P)$ generated by $\G_n$ is obtained by first including all comparable pairs $i<_P j$ if $i<j$ (order in natural numbers) and $ij\in E(\G_n)$, and then taking the transitive closure. We denote this partial order by $P_{\G_n}$. In other words, $P_{\G_n}$ consists of elements in $[n]$, where two elements $i$ and $j$ are comparable with $i<_P j$ if there exists a sequence of vertices $u_1,\ldots,u_k$ such that $u_1=i$, $u_k=j$, $u_1<u_2<\cdots<u_k$ and $u_1,\ldots,u_k$ is a path in $\G_n$. Barak and Erd\H{o}s~\cite{barak1984maximal} introduced the model $P_{\G_n}$ for $\G_n=\G(n,p)$, the Erd\H{o}s-R\'{e}nyi random graph, and they estimated the width of $P_{\G(n,p)}$ when $p$ is a fixed constant between 0 and 1. Then, Albert and Frieze~\cite{albert1989random} studied the height, width, dimension and the first-order logic of $P_{\G(n,p)}$ for fixed $p$. Alon, Bollob\'{a}s and Janson~\cite{alon1994linear}
 studied the number of linear extensions of $P_{\G(n,p=1/2)}$. The structures of sparse $P_{\G(n,p)}$ when $p=o(1)$ are studied by Bollob'{a}s and Brightwell~\cite{bollobas1997structure,bollobas1997dimension,bollobas1995width} including the anti-chains, the posts, and the dimension. The model $P_{\G_n}$ for $\G_n=\B(n,n,p)$ was introduced by Erd\"{o}s, Kierstad and Trotter~\cite{erdos1991dimension}, where $\B(n,n,p)$ denotes the random bipartite graph on parts $A=[n]$ and $B=\{n+1,\ldots,2n\}$ where $ij$ is an edge with probability $p$ independently for every $(i,j)\in A\times B$. The poset $P_{\B(n,n,p)}$ has a simpler structure than $P_{\G(n,p)}$: two elements $i$ and $j$ are related by $i\le_P j$ if and only if $i\in A$, $j\in B$ and $ij$ is an edge in $\B(n,n,p)$. In other words, adjacency of vertices in $\B(n,n,p)$ corresponds precisely to the order relations in $P_{\B(n,n,p)}$. 
The authors of~\cite{erdos1991dimension} studied the dimension of $P_{\B(n,n,p)}$, and proved the following.  
\begin{thm}[Erd\"{o}s, Kierstad and Trotter] \label{thm:Erdos}
For every $\eps>0$ there exists $\delta>0$ such that for all $p$ satisfying
\[
\frac{\log^{1+\eps} n}{n} <p< 1-n^{-1+\eps},
\]
a.a.s.\ $\dim P_{\B(n,n,p)} > (\delta pn\log pn)/(1+\delta p\log pn)$.
\end{thm}

Their lower bound is tight up to a logarithmic factor, by the upper bound given by F{\"u}rhedi and Kahn~\cite{furedi1986dimensions}; see Theorem~\ref{thm:Kahn} below.
In the same paper~\cite{erdos1991dimension}, they also asked many questions regarding $\dim P_{\B(n,n,p)}$, such as whether $\dim P_{\B(n,n,p)}$ is unimodal as $p$ increases? When does $\dim P_{\B(n,n,p)}$ become at least $k$ for constant $k$? What happens to $\dim P_{\B(n,n,p)}$ when $p$ is very close to 1? These questions remained open. As a comparison, Bollob\'{a}s and Brightwell~\cite{bollobas1997dimension} proved the following regarding $\dim P_{\G(n,p)}$, improving earlier bounds given by Albert and Frieze~\cite{albert1989random}, and extending the range of $p$ under study.

\begin{thm}[Bollob\'{a}s and Brightwell]\label{thm:BB}  The following hold a.a.s.\ for $\dim P_{\G(n,p)}$:
\begin{enumerate}[(a)]
\item If $1/\log\log n\ll p\le 1-1/\sqrt{\log n}$ then for every $\eps>0$ 
\[
(1-\eps)\sqrt{\frac{\log n}{\log(1/(1-p))}}\le \dim P_{\G(n,p)} \le (1+\eps)\sqrt{\frac{4\log n}{3\log(1/(1-p))}}.
\]
\item If $\log n/n\le p\ll 1/\log n$ then $\dim P_{\G(n,p)}=\Theta(p^{-1})$.
\item If $\tfrac{4}{5}\log n/n\le p\le (\log n-\log\log n)/n$ then there exist fixed $c_1,c_2>0$ such that 
\[
c_1e^{pn}\le \dim P_{\G(n,p)} \le c_2 e^{pn}\log^3 n.
\]
\end{enumerate}
\end{thm}

In view of Theorem~\ref{thm:Erdos}, there is a phase transition around $p=1/\log n$. Afterwards, for $1/\log n\ll p < 1-n^{-1+\eps}$, $\dim P_{\B(n,n,p)}\sim n$. There must be at least another phase transition at some $p > 1-n^{-1+\eps}$ since $\dim P_{\B(n,n,p)}$ eventually drops back to two.  In this paper, we capture this phase transition by determining the asymptotic order of the dimension of $P_{\B(n,n,p)}$ for all $p\ge 1- n^{-1+1/7}$,  addressing a question (see~\cite[Open problem (5.8)]{erdos1991dimension}) of Erd\H{o}s, Kierstead and Trotter: see Theorem~\ref{thm:bipartiteLargep} below. Theorem~\ref{thm:BB} shows that the behaviour of 
$\dim P_{\G(n,p)}$ is quite different from $\dim P_{\B(n,n,p)}$: there are phase transitions around $p=\log n/n$ and then around $p=1/\log n$. Bollob\'{a}s and Brightwell predicted ``a third transition phase, as $p$ decreases through $c/n$, since the underlying random graph changes from being almost all in one component, to being a collection of small trees''. It is known that $\dim P_{\G(n,p=c/n)}\le 4$ when $c<1$ (see Theorems~\ref{thm:Trotter} and~\ref{thm:Brightwell} below), since $\G(n,p=c/n)$ is a collection of small trees and unicyclic components. When $c>1$, $\G(n,p=c/n)$ contains a giant component whose order is linear in $n$, and has many cycles and other complicated structures. For instance, for every fixed $k\ge 1$, $\G(n,p=c/n)$ has a subgraph with minimum degree at least $k$, provided that $c$ is sufficiently large (depending on $k$). Does ``a third transition phase'' appears in the way that right after $c$ passes 1, $\dim P_{\G(n,p)}$ immediately blows up to $\omega(1)$? Indeed, Bollob\'{a}s and Brightwell specifically asked in~\cite{bollobas1997dimension} what happens to $\dim P_{\G(n,p)}$ if $p=c/n$ with fixed $c>1$.

In this paper we address the question of this predicted ``third transition phase'', and in particular, whether the phase transition of random graph structure from a collection of small tree or unicyclic components to the appearance of a giant component is associated with a sudden jump of the dimension of the random graph order from at most four to $\omega(1)$. Rather surprisingly, 
our result is negative: see Theorems~\ref{thm:mainbipartite} -- \ref{thm:mainnonbipartite} in Section~\ref{sec:results}.

\subsection{Preliminary and notations}

Let $P$ be a poset on a finite set of elements $V$. Without loss of generality we denote $V$ by $[n]=\{1,2,\ldots,n\}$ where $n$ is the number of elements in $V$.  We always use  $<_P$  to denote the order in $P$, and use $<$ to denote the order in natural numbers. For $x,y\in V$, we use $x||y$ to denote that $x$ and $y$ are incomparable in $P$.

A realiser of a partial order $P$ on a finite set of elements $V$ is a set of total orders $\R=\{L_1,\ldots, L_k\}$ on $V$ satisfying the following conditions.
\begin{enumerate}[(a)]
    \item if $i<_P j$ in $P$ then $i<_{L}j$ for every $L\in \R$; and
    \item if $i || j$ in $P$ then $i<_{L}j$ and $j<_{L'}i$ for some $L,L'\in\R$.
\end{enumerate}
The dimension of $P$ is defined by the size of a minimum realiser of $P$.

We say a poset $P$ is bipartite if the set of elements can be partitioned into two parts $A$ and $B$ such that if $x<_P y$ then $x\in A$ and $y\in B$, and we use $P(A,B)$ to denote a bipartite poset, where $A,B$ is a bipartition of the elements. Given a poset $P$ and a subset $S$ of the elements in $P$, let $P[S]$ be the sub-poset of $P$ induced by $S$, i.e.\ two elements $u,v$ in $S$ have the same relation in $P[S]$ as they are in $P$.

Given a poset $P$, the dual of $P$ is defined on the same set of elements of $P$, such that for every $x\neq y$, $x\parallel y$ in $P'$ if and only if $x\parallel y$ in $P$, and $x<y$ in $P'$ if and only if $y<x$ in $P$. Obviously, $\dim P'=\dim P$.

Given a poset $P$, the comparable graph of $P$ is the graph defined on $V(P)$, the set of elements in $P$, where two elements are adjacent if they are comparable in $P$. On the other hand, the cover graph (also called the Hasse diagram) of $P$ is the graph defined on $V(P)$ such that $xy$ are adjacent in the cover graph if $x\neq y$, $x<_P y$ and there is no $z$ distinct from $x$ and $y$ such that $x<_P z <_P y$. Trotter and Moore proved that if $P$ has a cover graph that is isomorphic to a tree then the dimension of $P$ is at most three. 

\begin{thm}[\cite{trotter1977dimension}]\label{thm:Trotter}
Let $P$ be a poset whose cover graph is a tree. Then, $\dim P \leq 3$.
\end{thm}

How about poset whose cover graph has a unique cycle? Bollob\'{a}s and Brightwell~\cite{bollobas1997dimension} conjectured that such posets have dimension at most three. Abram and Segovia recently reported solving this conjecture~\cite{AS}. The best known upper bound before~\cite{AS} is four, which follows from the upper bound on the dimension of a poset whose cover graph is outerplanar~\cite{felsner2015dimension}. There is also a short proof provided by Brightwell (from personal communications). Since his proof is short and self-contained, we include it below for completeness. Given a total ordering $L$ on a set of elements $V$ and $S\subseteq V$, let $L|_S$ denote the total ordering on $S$ obtained by restricting $L$ on $S$. Given a subset of elements $S$ of $P$, let $U_P[S]$ and $D_P[S]$ denote the up-set and the down-set of $S$, i.e.\ 
\begin{equation}U_P[S]=\{x:\ y\le_P x\ \text{for some $y\in S$}\},\quad D_P[S]=\{x:\ x\le_P y\ \text{for some $y\in S$}\}.\label{def:UD}
\end{equation}We may drop the subscript $P$ when it is clear from the context. If $S=\{x\}$ we write $U[x]$ for $U[\{x\}]$, and we write $U(x)$ for $U[x]\setminus \{x\}$.

\begin{thm}\label{thm:Brightwell}

Let $P$ be a poset with unicyclic cover graph. Then, $\dim P \leq 4$.
\end{thm}

\begin{proof}
This is a proof provided by Brightwell (via personal communications). Let $G$ be the unicycle cover graph, and let $C$ be the cycle in $G$.
Let $x$ be a maximal element in $C$. Let $A$ be the set of neighours of $x$ in $G$ that are greater than $x$ in $P$, and let $u_1,\ldots,u_k$ be an enumeration of the vertices in $A$. Let $W_i$ be the component of $G-x$ containing $u_i$ for $1\le i\le k$. Let $W[x]=\{x\}\cup (\cup_{i=1}^k W_i)$, and let $R=V(G)\setminus W[x]$. It is readily seen that both the cover graphs of the subposets of $P[W[x]]$ and $P[R]$ are acyclic, and hence each of them has dimension at most three. Moreover, $U(x)\subseteq W[x]$ and $D(x)\subseteq R$. Let $\{\sigma_1,\sigma_2,\sigma_3\}$ be a realiser for  $P[W[x]]$ and $\{\tau_1,\tau_2,\tau_3\}$ be a realiser for  $P[R]$. Let $\hat L$ be an arbitrary total ordering on the elements in $P$. Construct four linear orderigs on the set of elements in $P$ as follows.
\begin{align*}
&\sigma_i |_{R} <_{\pi_i} \tau_i |_{W[x]} \quad \text{for $1\le i\le 3$}, \\
&\hat L |_{W[x]\setminus {U[x]}} <_{\pi_4} \hat L|_{D(x)} <_{\pi_4} \hat L|_{U[x]} <_{\pi_4} \hat L |_{R\setminus D(x)}. 
\end{align*}
Note that the notation in the first inequality above means that first ordering elements in $R$ according to $\sigma_i$, followed by elements in $W[x]$ ordered according to $\tau_i$.
It is straightforward to verify that $\{\pi_i:\ 1\le i\le 4\}$ is a realiser for $P$.
\end{proof}

For a general poset, the following upper bound on its dimension given by F{\"u}rhedi and Kahn~\cite{furedi1986dimensions} is the best known, and has been proved to be close to being optimal~\cite{erdos1991dimension}.

\begin{thm}[F{\"u}rhedi-Kahn~\cite{furedi1986dimensions}]\label{thm:Kahn}
    Suppose that $P$ is a poset in which every element is comparable to at most $k$ other elements. Then, $\dim P=O(k\log^2 k)$. 
\end{thm}
The following lemma was used in the proof in~\cite{furedi1986dimensions}, which is useful in this paper as well. Note that ${\R}$ in the lemma below is not necessarily a realiser. The idea is that some operation called the ``left-shift'' can be repeatedly applied to turn $\R$ into a realiser. We refer interested readers to~\cite{furedi1986dimensions} for more details.

\begin{lemma}[\cite{furedi1986dimensions}]\label{lem:Kahn}
    Let $P$ be a poset on $V$, and let $\R$ be a set of total orderings of $V$. Suppose that for every $x,y\in V$ such that $x\nleq_P y$ there exists $L\in \R$ such that $y\le_L U[x]$. Then $\dim P\le |\R|$.
\end{lemma}

\subsection{Main results}
\label{sec:results}

There are two well known upper bounds for the dimension of a poset $P$. The first is the width of $P$, which is the size of a largest anti-chain. The second upper bound is given by
F{\"u}rhedi and Kahn~\cite{furedi1986dimensions} in terms of the maximum number of elements comparable to a given element in the set (Theorem~\ref{thm:Kahn}).  As mentioned earlier, this bound is almost optimal in view of Theorem~\ref{thm:Erdos}, since it matches, up to a logarithmic factor, the lower bound on the dimension of $P_{\B(n,n,p)}$, when $p\ge \log^{1+\eps} n/n$ (and $p$ is not too close to 1). Bollob\'{a}s and Brightwell proved a similar phenomenon in  $P_{\G(n,p)}$ for $p=\Omega(\log n/n)$:  the dimension is close to the minimum of the width, and the F{\"u}rhedi-Kahn bound (Theorem~\ref{thm:BB}).  This is no longer the case for sparser random graph orders. Consider $P_{\G(n,p=c/n)}$ for an example. The minimum of the above two bounds would be around $\log n\log^2 \log n$. However, although the maximum number of comparable elements to a given element can be as large as $\log n$, most elements are only comparable to $O(e^c)$ other elements, and the number of ``problematic'' elements who are comparable to more than $ce^{Kc}$ other elements are bounded by $e^{-Kc} n$ (see Corollary~\ref{cor:tools}). It is then reasonable to have a careful study of the sub-posets induced by these  ``problematic'' elements, and study how they impact the dimension of the original poset. We prove the following theorems which upper bounds the dimension of a poset by the sum of certain sub-posets. Although these bounds are specifically tailored to help to bound the dimension of sparse random graph orders, we believe that they are of independent interest, as we are not aware of bounds of this nature in the literature.

Our first result is specifically for bipartite poset $P(A,B)$ where the elements in $B$ are greater than the elements in $A$.

\begin{thm}\label{thm:deterministicBipartitie}
Let $P=P(A,B)$ be a bipartite poset where $U[A]\setminus A\subseteq B$.  Let $S_A\subseteq A$ and $S_B\subseteq B$. Then,
\[
\dim P\le \dim P[U[S_A]]+\dim P[A\cup B\setminus S_A];
\]
\[
\dim P\le \dim P[D[S_B]]+\dim P[A\cup B\setminus S_B].
\]
\end{thm}

For a general poset we have the following upper bound for its dimension.

\begin{thm}\label{thm:deterministicNonBipartitie}
Let $P$ be a poset on $V$ and suppose that $S\subseteq V$. Then,
\begin{align*}
&\dim P  \le 2(\dim P[D[S]\cup U[S]]+\dim P[V\setminus S] ); \\
&\dim P[D[S]\cup U[S]] \le \dim P[D[S]]+\dim P[U[S]]+\dim P[D\cup U],
\end{align*}
where $U=U[S]\setminus S$ and $D=D[S]\setminus S$.
\end{thm}

Note that $\dim P(A,B)\ge \max\{\dim P[U[S_A]],\dim P[A\cup B\setminus S_A]\}$ in the bipartite case, and $\dim P\ge \max\{\dim P[D[S]\cup U[S]],\dim P[V\setminus S] \}$ in the general case, the upper bounds in the above two theorems are tight up to a constant factor.

Now we turn to random graph orders. We start from addressing a question of Erd\H{o}s, Kierstead and Trotter regarding the dimension of $P_{\B(n,n,p)}$ when $p$ is very close to 1.

\begin{thm}\label{thm:bipartiteLargep} Let $q:=1-p$. 
\begin{enumerate}[(a)]
\item If $q=O(1/n)$ and $q=\omega(1/n^2)$ then a.a.s.\ $\dim P_{\B(n,n,p)}=\Theta(n^2q)$, and $n-\dim P_{\B(n,n,p)}=\Omega(n)$.
\item If $q=\omega(1/n)$ and $q\le n^{-1+1/7}$ then a.a.s. $\dim P_{\B(n,n,p)}=(1-o(1))n$.
\end{enumerate}
\end{thm}

\noindent {\bf Remark.} The vanishing error term $o(1)$ in part (b) (cf.~the $O(1/\log n)$ error in Theorem~\ref{thm:Erdos}) can in fact be made explicit from our proof, although we did not pursue such a quantification. Our main message here is that there is a phase transition from $q=\omega(1/n)$ to $q=O(1/n)$, at which point the dimension of the bipartite partial order ceases to be asymptotic to $n$. We excluded the uninteresting case $q=O(1/n^2)$ from part (a) since in that regime $\bar B$ contains $O_p(1)$ isolated edges and immediately $\dim P_{\B(n,n,p)}=O_p(1)$.

Next we turn to the sparse regime. Using Theorem~\ref{thm:deterministicBipartitie}
we prove the following bounds on the dimension of $P_{\B(n,n,p)}$ where $p=c_n/n$. We assume that $c_n\ge 1$ and $c_n\le \log^2 n$, since the dimension of $P_{\B(n,n,p)}$ for $p<1/n$ is well understood by Theorem~\ref{thm:Trotter}, and the case $p=\Omega(\log^2 n/n)$ has already been studied by Erd\"{o}s, Kierstad and Trotter~\cite{erdos1991dimension}.

Given $c>0$, let $x=\alpha_c$ be the unique positive solution to 
\begin{equation}
\frac{e^2}{x^2}e^{-cx}=1. \label{def:alpha_c}
\end{equation}

\begin{thm}\label{thm:mainbipartite}
    Let $p=c_n/n$ where $c_n\ge 2$ and $c_n\le \log^2 n$. Then, 
    \begin{enumerate}[(a)]
    \item a.a.s.\ $\dim P_{\B(n,n,p)}\ge (1-o(1))\frac{1}{2\alpha_{c_n}}$;
    \item there exist absolute constants $\gamma_1$ and $\gamma_2$, both independent of $n$ and $c_n$ such that a.a.s.\ $\gamma_1 c_n/\log c_n\le \dim P_{\B(n,n,p)}\le \gamma_2 c_n\log^2 c_n$.
    \end{enumerate}
\end{thm}

\noindent {\bf Remark}.  $\alpha_{c_n}\sim 2\log c_n/c_n$ as $c_n\to\infty$ (see Lemma~\ref{lem:asymptotic}), and thus $\dim P_{\B(n,n,p=c/n)}\ge (1-o_c(1)) c/4\log c$ as $c\to \infty$.
On the other hand, for small values of $c_n$, $\alpha_{c_n}$ can be numerically computed. The figure below on the left  is a plot of $1/2\alpha_c$ for $2<c<30$. From the plot the reader can see that the lower bounds obtained for small values of $c$ (below 20) are smaller than the trivial lower bound two. We stress that the merit of Theorem~\ref{thm:mainbipartite}, together with the upper and lower bounds for $\dim P_{\G(n,p)}$ in Theorems~\ref{thm:nonbipartiteLB} and~\ref{thm:mainnonbipartite}, is about the asymptotic behaviour of $\dim P_{\B(n,n,p)}$ when $c$ goes large. In fact, the implicit constant coefficients in the upper bounds (in  both Theorems~\ref{thm:mainbipartite} and~\ref{thm:mainnonbipartite}) are inherited from the implicit constant in the F\"{u}redi-Kahn bound, and thus these bound do not yield useful information when $c$ is small. Obtaining nontrivial lower or upper bounds on $P_{\B(n,n,p)}$ for small values of $c$ remains open. See Open Problem~\ref{open} below.

\noindent\begin{minipage}{0.4\textwidth}
    \includegraphics[width=\linewidth]{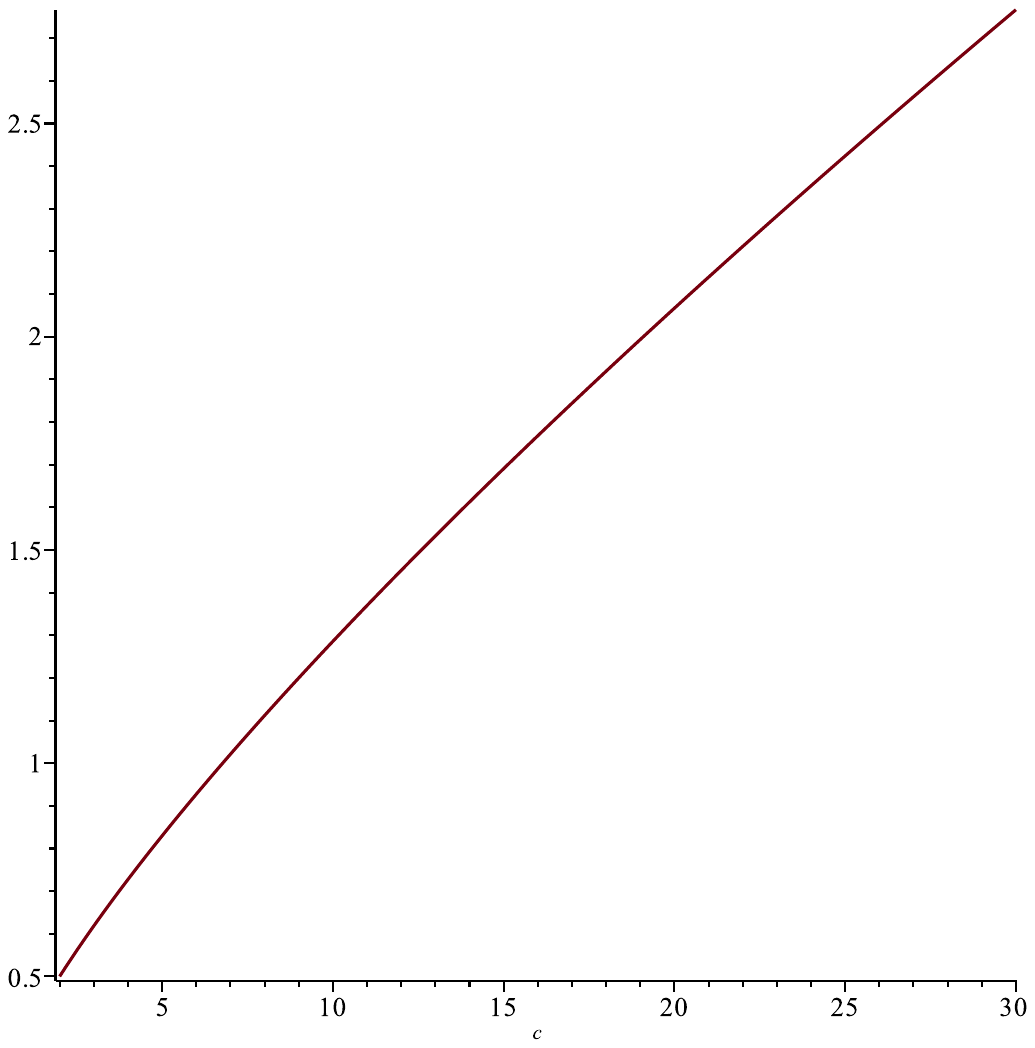}
\end{minipage}%
\hfill%
\begin{minipage}{0.4\textwidth} 
\includegraphics[width=\linewidth]{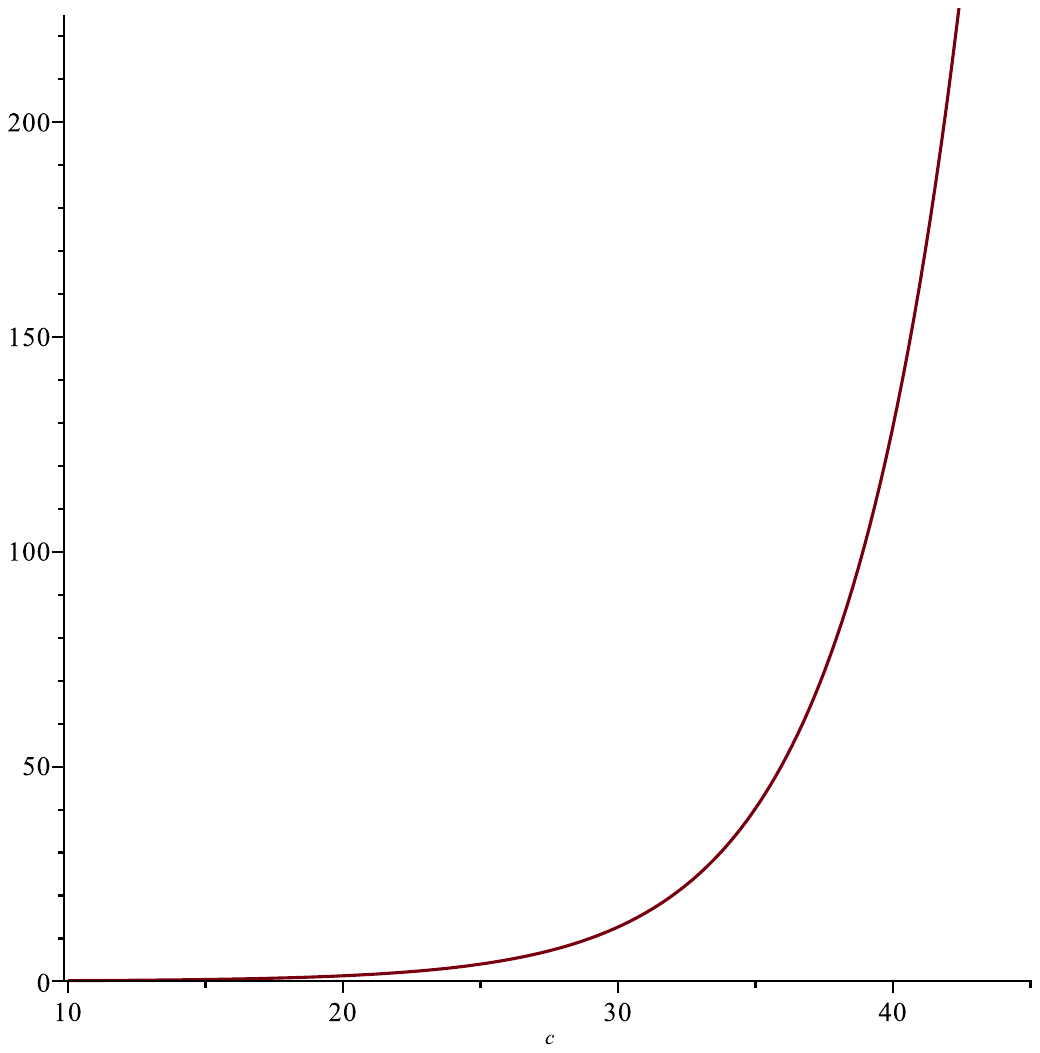}
\end{minipage}

\vspace{-2cm}

Next we consider the dimension of $P_{\G(n,p)}$.  Let $Li_2(z):=-\int_{0}^z y^{-1}\log(1-y) dy$ denote the dilogarithm function, also known as the Spencer's function. 
\begin{thm}\label{thm:nonbipartiteLB} Let $c>0$ be fixed and let $p=c/n$.
Suppose that $\xi, \beta$ are fixed positive real numbers such that $\xi+\alpha_c<1/2$ and $\sup_{t\le \alpha_c}f(c,\xi,t,\beta)<0$, where 
\begin{align}
f(c,\xi,\beta,t)=&\beta\log(e\xi/\beta)-c\beta (1/2-\xi-t)+\frac{1}{c}\left(Li_2(1)- Li_2(e^{-ct} )\right)\nonumber\\
&+\frac{1}{c}\left(Li_2(e^{-c\beta-ct})-Li_2(e^{-c\beta}) + Li_2(e^{-c(1/2-\xi-t)-ct})-Li_2(e^{-c(1/2-\xi-t)}) )\right).\label{def:f}
\end{align}
Then, 
\begin{enumerate}[(a)]
\item a.a.s.\  $\dim P_{\G(n,p)}\ge \xi/2\beta$;
\item there exist constants $\gamma:=\gamma(c)>0$ where $\lim_{c\to\infty} \gamma(c)=1/6$ such that a.a.s.\ $\dim P_{\G(n,p)}\ge e^{\gamma c}$.
\end{enumerate}
\end{thm}

\noindent {\bf Remark} (a). To obtain Theorem~\ref{thm:nonbipartiteLB}(b) we applied part (a) with the choice $(\xi,t)=(\frac{1}{\log c}, \frac{3\log c}{c})$, which works well for large $c$. For small values of $c$, a lower bound  in Theorem~\ref{thm:nonbipartiteLB}(a) can be numerically computed by choosing some suitable $\xi$. We plotted the resulting lower bound on the figure above on the right, by choosing $\xi=1/15\log c$, where $c$ is between 10 and 50. However we did not try to optimize over the choice of $\xi$. 

(b). The assumption that $c=O(1)$ is imposed for technical reason only. In the proof we approximate a Riemann sum by the dilogarithm function. The lack of uniform bound on the derivative of $y^{-1}\log(1-y) $ prevents us from extending the proof for $c_n$ as a growing function of $n$.  

For the upper bound on $\dim P_{\G(n,p)}$ we have the following result.

\begin{thm}\label{thm:mainnonbipartite}
    Let $p=c_n/n$ where $c_n\ge 1$ and $c_n=o(\log n)$. Then, a.a.s.\ $\dim P_{\G(n,p)}=O(c_n^4 e^{c_n})$. 
\end{thm}

Notice that the coefficient in the exponent for the lower bound on $\dim P_{\G(n,p)}$ in Theorem~\ref{thm:nonbipartiteLB} does not match that in the upper bound in Theorem~\ref{thm:mainnonbipartite} for large $c$. We conjecture that our upper bound is essentially tight.

\begin{con}
Let $p=c_n/n$ where $1\le c_n\le \log n$. Then, there exist $\beta(c),\gamma(c)>0$ where $\lim_{c\to\infty}\beta(c)=\lim_{c\to\infty}\gamma(c)=1$ such that a.a.s.\ $e^{\beta(c_n) c_n}\le \dim P_{\G(n,p)}\le e^{\gamma(c_n) c_n}$.
\end{con}

\noindent {\bf Remark.} Our Theorem~\ref{thm:mainnonbipartite} confirms the upper bound in the conjecture for $c_n=o(\log n)$.  Bollob\'{a}s and Brightwell conjectured that Theorem~\ref{thm:BB}(c) holds if $\frac{4}{5}$ is replaced by an arbitrarily small constant $\eps>0$.

Our results show that the dimensions of $P_{\B(n,n,p)}$ and $P_{\G(n,p)}$ does not undergo a sudden change when $p$ passes $1/n$. However our results also do not exclude the possibility that there is a ``non-smooth'' change at $p=1/n$. We leave it as an open problem.

\begin{question}\label{open}
What is the largest positive integer $I$ that satisfies the following property: for all fixed $c>1$, a.a.s.\ $\dim P_{\G_n} \ge I$, where $\G_n\sim \G(n,p=c/n)$ or $\G_n\sim \B(n,n,p=c/n)$?
\end{question}



We use standard Landau notation throughout the paper, with asymptotics understood to be taken as $n\to\infty$ unless otherwise specified. When the asymptotic behaviour depends on a different variable --- such as $c$ --- we indicate this by including the variable in the subscript to emphasize the dependence. For instance, $o(1)$ denotes some real number which goes to zero as $n\to\infty$, whereas $o_c(1)$ denotes some real number that goes to zero as $c\to \infty$. 
We say an event holds asymptotically almost surely (a.a.s.) if its probability goes to 1 as $n\to\infty$.

\subsection{Connections to causal set theory}

Causal set theory is a theoretical framework that models spacetime as a fundamentally discrete, combinatorial structure. In contrast to general relativity, which treats spacetime as a smooth manifold, this approach---motivated by ideas from quantum gravity and quantum mechanics---posits that spacetime consists of indivisible ``atoms'' of events, organized into a locally finite poset; see~\cite{sorkin1991spacetime}. The partial order in this poset encodes the causal relationships between events.
A central challenge in causal set theory is to generate such posets in a way that is both physically meaningful and mathematically rigorous. One of the most influential approaches to this problem is the classical sequential growth (CSG) model. Among the models studied, the random graph model 
$P_{\G(n,p)}$	
 serves as a primary example of a CSG model within this framework; see~\cite{brightwell2016mathematics} for work connecting the studies of $P_{\G(n,p)}$ with developments in causal set theory.

A key concept in the study of causal sets is the causal representation of a poset in Minkowski space. Roughly speaking, the goal is to embed a poset 
$P$ into the $d$-dimensional Minkowski space such that the order relation corresponds to the lightcone containment---that is, $x<y$ in $P$ if and only if the event $x$
 lies in the causal past of $y$ in the embedding. The Minkowski dimension of $P$ is the smallest dimension $d$ for which such a representation exists. We refer interested readers to~\cite{meyer1993spherical} for more details.
While both the (order-theoretic) dimension and the Minkowski dimension measure the complexity of a poset, they do so in fundamentally different ways. Studying the Minkowski dimension of the random poset $P_{\G(n,p)}$ remains an intriguing and open problem, linking combinatorics, geometry, and the foundations of spacetime.


\section{Proof of Theorem~\ref{thm:deterministicBipartitie}.\ }

For a bipartite poset $P=P(A,B)$, it is convenient to denote  $U[S]\setminus S$ 
by $U(S)$ for $S\subseteq A$, and similarly denote $D[S]\setminus S$ by $D(S)$ for $S\subseteq B$. Note that $U(S)\subseteq B$ for $S\subseteq A$ and $D(S)\subseteq A$ for $S\subseteq B$.


Let $P_1=P[U[S_A]]$ and $P_2=P[A\cup B\setminus S_A]$, and let $d_i=\dim P_i$ for $i\in \{1,2\}$.
Let $\{ L_1^i, \ldots, L_{d_i}^i \}$ be a realiser for $P_i$. Let $\hat L$ be an arbitrary total ordering on $A\cup B$. Let $A_1=S_A$, $A_2=A\setminus S_A$, $B_1=U(A_1)$ and $B_2=U(A_2)$.

Let $\pi_j^i$ be a permutation on $A\cup B$ defined by

\begin{equation} \hat L|_{A \setminus A_i} <_{\pi_j^i}  \hat L|_{B \setminus B_i}   <_{\pi_j^i} L_j^i\quad \text{for all $i \in [2], j \in [d_i]$.}\label{eq:realiser1}
\end{equation}


\begin{claim}\label{claim:reverse}
   For all  $x, y \in P$ where $x \nleq_P y$, there exists $i \in [2]$ and $j \in [d_i]$ such that $y \leq_{\pi_j^i} U(x)$.
\end{claim}
Now the first assertion of Theorem~\ref{thm:deterministicBipartitie} follows by Lemma~\ref{lem:Kahn} and Claim~\ref{claim:reverse}. For the second assertion, consider $P'$, the dual of $P$. 
Then, $U_{P'}(S_B)=D_{P}(S_B)\subseteq A$ and $U_{P'}(B\setminus S_B)=D_{P}(B\setminus S_B)\subseteq A$. Thus, the second assertion follows from the first assertion of Theorem~\ref{thm:deterministicBipartitie} by symmetry.

It only remains to prove Claim~\ref{claim:reverse}. \smallskip

\noindent {\em Proof of Claim~\ref{claim:reverse}.\ } 
    Let $x, y$ be elements in $ P$ with $x \nleq_P y$.  We consider two cases.
    
    {\em Case 1:} Suppose that there exists $X\in [2]$ such that $x,y \in P_X$. Then $x \nleq_P y$ implies  that there is some $j \in [d_X]$ for which $y \leq_{L_j^X} x$, which implies that $y\le_{L_j^X} U(x)$, since $L_j^X$ is a realiser of $P_X$ and moreover $P_X$ contains all elements in $U_P(x)$ by the definition of $P_1$ and $P_2$.
  Thus, $y \leq_{\pi_j^X} U(x)$, as desired. \smallskip
  
{\em Case 2:} 
        Suppose that $x \in P_X$ and $y \in P_Y$ where $X, Y \in [2]$,
  and $X \neq Y$. Then as before $P_X$ contains all elements in $U_P(x)$. Moreover, we may assume that $y \in A \setminus A_X$ as otherwise $y$ would be in $P_X$ and we would have considered it as in case 1. It follows now that $y \leq_{\pi_1^X} U(x)$.
 \qed
 
 \qed

\section{Proof of Theorem~\ref{thm:deterministicNonBipartitie}}
We prove the first inequality.
Let $U=U[S]\setminus S$, $D=D[S]\setminus S$ and $W=V\setminus (U\cup S\cup D)$. Then $D[S]\cup U[S]=U\cup S\cup D$. It follows that $W$ is exactly the set of elements $y$ such that $y\parallel x$ for every $x\in S$. Let $d_1=\dim P[U\cup S\cup D]$ and $d_2=\dim P[V\setminus S]$. Let $\{L^1_j\}_{1\le j\le d_1}$ be a realiser of $P[U\cup S\cup D]$, and $\{L^2_j\}_{1\le j\le d_2}$ be a realiser of $P[V\setminus S]$. Construct total orderings $\pi$, $\pi'$, $\pi''$ and $\pi'''$  on $V$ as follows:





\[
L^2_1 |_W <_{\pi_{j}} L^1_j, \quad j\in [d_1] 
\]

\[
 L^1_j <_{\pi'_{j}} L^2_1 |_W , \quad  j\in [d_1] 
\]

\[
L^1_1 |_S<_{\pi''_i}L_i^2\quad i\in[d_2]
\]

\[
L_i^2 <_{\pi'''_i}  L^1_1 |_S\quad i\in[d_2].
\]

Let ${\mathcal R}=\{\pi_{j},\  j\in [d_1]\}\cup \{\pi'_{j},\  j\in [d_1]\}\cup \{\pi''_{i},\ i\in[d_2]\}\cup \{\pi'''_{i},\ i\in[d_2]\}$. Note that $|{\mathcal R}|=2(d_1+d_2)$. Now Theorem~\ref{thm:deterministicNonBipartitie} follows by Theorem~\ref{lem:Kahn} and the following claim.

\begin{claim}\label{claim:realiser}
For every $x,y\in V$ such that $x\nleq_P y$ there exists $L\in {\mathcal R}$ such that $y\le_L U_P[x]$.
\end{claim}

\noindent {\em Proof of Claim~\ref{claim:realiser}.\ } There are three cases of $x,y$ such that $x\nleq_P y$ in terms of the location of $x$: $x\in S\cup U$; $x\in W$ and $x\in D\setminus U$. We identify $L\in {\mathcal R}$ such that $y\le_L U_P[x]$ in each case below.

{\em Case 1:} $x\in S\cup U$.  Then, $U_P[x]\subseteq U\cup S\cup D$. If $y\in U\cup S\cup D$, then, since  $\{L^1_j\}_{1\le j\le d_1}$ is a realiser of $P[U\cup S\cup D]$, there must exists $j$ such that $y<_{L^1_j} U_P[x]$ and hence  $y<_{\pi_{j}} U_P[x]$. If $y\in W$ then obviously, $y<_{\pi_{1}} U_P[x]$.  

{\em Case 2}: $x\in W$. In this case, $U_P[x]\subseteq U\cup W$. If $y\in S$, then $y<_{\pi_1''} U_P[x]$. If $y\in V\setminus S$,
 then, then there exists $i$ such that $y<_{L^2_i} U_P[x]$, and hence, $y<_{\pi''_i} U_P[x]$.


{\em Case 3:} $x\in D\setminus U$. In this case, $U_P[x]$ may intersect all of $S,U,D,W$. We consider two subcases. If $y\in W$, then there exists $i$ such that $y<_{L^2_i} U[x]\setminus S$, and hence, $y<_{\pi'''_i}U_P[x]$. If $y\notin W$, then there exists $j$ such that $y<_{L^1_j} U_P[x]\setminus W$, and hence $y<_{\pi'_{j}} U_P[x]$.
\qed \smallskip

The proof for the second inequality is similar. Let $d_1=\dim P[D[S]]$, $d_2=\dim P[U[S]]$ and $d_3=\dim P[U\cup D]$. Let $\{L^1_j\}_{1\le j\le d_1}$ be a realiser of $P[D[S]]$,  $\{L^2_j\}_{1\le j\le d_2}$ be a realiser for $P[U[S]]$, and $\{L^3_j\}_{1\le j\le d_3}$ be a realiser for $P[U\cup D]$. Construct total orderings $\pi,\pi',\pi''$ on $U\cup S\cup D$ as follows:
\begin{align*}
&L_1^1|_{D\setminus U}<_{\pi_j} L^2_j, \quad j\in [d_2]\\
&L_1^i<_{\pi'_i} L_2^1|_{U\setminus D}, \quad i\in [d_1]\\
&L^3_k|<_{\pi''_k} L_1^1|_S, \quad k\in [d_3].
\end{align*}
Let ${\mathcal R}=\{\pi_{j},\  j\in [d_2]\}\cup \{\pi'_{i},\  i\in [d_1]\}\cup \{\pi''_{k},\ k\in[d_3]\}$.  It suffices to show that Claim~\ref{claim:realiser} holds here as well.
For pairs of $(x,y)$ such that $x\nleq_P y$ we consider the two cases in terms of location of $x$: $x\in S\cup U$, and $x\in D\setminus (S\cup U)$. We say a permutation $L$ reverses $(x,y)$ if $y<_L U[x]$.

{\em Case 1: $x\in S\cup U$.} In this case, $U[x]\in S\cup U$. If $y\in S\cup U$, then some $\pi_j$ reverses $(x,y)$. If $y\in D\setminus (S\cup U)$, every $\pi'_i$ reverses $(x,y)$. 

{\em Case 2: $x\in D\setminus (S\cup U)$.} If $y\in D\cup S$ then some $\pi'_i$ reverses $(x,y)$. If $y\in U\setminus (D\cup S)$ then some $\pi''_k$ reverses $(x,y)$. 
\qed

 \section{Dimension of random bipartite POS}

\subsection{The sparse regime: proof of Theorem~\ref{thm:mainbipartite}}

In this section, let $P\sim P_{\B(n,n,p=c_n/n)}$ and $0<c_n\le \log^2 n$.  We first demonstrate the lower bound and then the upper bound.

\subsubsection{Lower bound}
\label{sec:bipartitleLB}

Our proof uses ideas in~\cite{erdos1991dimension} but is much simpler (the proof in~\cite{erdos1991dimension} aims for a greater lower bound of order $c_n\log c_n$; but that proof fails for the range of $c_n$ considered in this paper). Our aim is to show that a.a.s.
\[
\dim P\ge \gamma_1 c_n/\log c_n.
\]

\begin{lemma}\label{lem:asymptotic} $0<\alpha_c\le 1$ for all $c\ge 2$, and
$\alpha_c \sim 2\log c/c$ as $c\to\infty$.
\end{lemma}

\proof $\tfrac{e^2}{x^2}e^{-cx}$ is a decreasing function of $x$, and is equal to $e^{2-c}$ when $x=1$. Hence, $0<\alpha_c\le 1$ if $c\ge 2$.
 Fix $\eps>0$. Let $x=(2+\eps)\log c/c$ and $y=(2-\eps)\log c/c$. It follows that for all sufficiently large $c$,
\[
x^2e^{cx} \le \exp(4+2\log\log c -\eps \log c) <e^2; \quad y^2e^{cy} \ge \exp(2\log \log c+\eps \log c) >e^2.  
\]
Thus, $ \alpha_c \sim 2\log c/c$ as $c\to\infty$.
\qed


\begin{lemma} \label{lem:Bproperties} Let $\eps>0$ be fixed; $c_n\ge 2$ and $\alpha=(1+\eps)\alpha_{c_n}$.
    \begin{enumerate}[(a)]
        \item A.a.s.\ $P$ has $(1-o(1))n^2$ incomparable pairs $(i,j)\in A\times B$.
        \item A.a.s.\ for all $S \subseteq A$ and $T \subseteq B$ where $|S|,|T|\ge\alpha n$, there is an edge between $S$ and $T$ in $\B(n,n,p)$.
    \end{enumerate}
\end{lemma}

\proof Part (a) is trivial as the maximum degree of $\B(n,n,p)$ is $O(\log^2 n)$ if $p\le \log^2 n/n$. For part (b), the probability that there is a set $S \subseteq A$ and $T \subseteq B$ with size at least $\alpha n$ such that there is no edge between $S$ and $T$ is at most $${n \choose \alpha n}^2 (1-p)^{\alpha^2 n^2} \le \left( \frac{e}{\alpha} \right)^{2 \alpha n} e^{-p \alpha^2 n^2} = \left( \frac{e^2}{\alpha^2} e^{-c \alpha} \right)^{\alpha n}\le \exp\left(-c_n\eps \alpha^2_{c_n} n \right)=o(1),$$ where the last equality holds by Lemma~\ref{lem:asymptotic}.\qed

\medskip

Let $(i,j)\in A\times B$ be an incomparable pair of $P_{\B(n,n,p)}$. We say that a permutation $L$ on $[2n]$ reverses $(i,j)$ if $j<_L i$.
Using the above properties in Lemma~\ref{lem:Bproperties} we prove the following. 
\begin{lemma} \label{lem:reversible} Let $\eps>0$ be fixed. Suppose $c_n\ge 2$ and let $\alpha=(1+\eps)\alpha_{c_n}$.
 A.a.s.\ every linear extension $L$ of $P$ reverses at most $2\alpha n^2$ incomparable pairs of  $P$.
\end{lemma}

\proof Let $L$ be a linear extension of $P$. Let $\overline{A}$ denote the greatest $\alpha n$ elements in $A$, and let $\underline{B}$ denote the smallest $\alpha n$ elements in $B$ with respect to the total order $L$. By property (b) of Lemma~\ref{lem:Bproperties}, there exist $i\in \overline{A}$ and $j\in \underline{B}$ such that $i<_P j$. 
Let $F$ be the set of incomparable pairs that are reversed by $L$. Suppose $(x,y)\in F$. Then, either $x\ge_L i$ or $y\le_L j$, as otherwise the edge $ij$ in $\B(n,n,p)$ would imply that $x\le_L i<_L j\le_L y$, contradicting with $(x,y)\in F$. Consequently, for every $(x,y)\in F$, either $x\in \overline{A}$ or $y\in \underline{B}$. It follows then that $|F|\le |\overline{A}|n+|\underline{B}|n\le 2\alpha n^2$. \qed
\smallskip

To prove the lower bound in Theorem~\ref{thm:mainbipartite}, notice that by Lemma~\ref{lem:Bproperties}(a) and Lemma~\ref{lem:reversible}, for any fixed $\eps>0$, a.a.s.\ $$\dim P\ge \frac{(1-o(1))n^2}{2 \alpha n^2} \sim \frac{1}{2 \alpha}$$ where $\alpha=(1+\eps)\alpha_c$. This yields our desired bound in part (a). The lower bound in part (b) follows by Lemma~\ref{lem:asymptotic}.

\subsubsection{Upper bound}

Let 
\begin{equation}
\alpha=Kc_n\label{def:alpha}
\end{equation}
where $K>0$ is a fixed constant (independent of $c_n$) that is to be specified later.
Let $S_A=\{v\in A: \deg(v)\ge \alpha\}$ and $S_B=\{v\in B: \deg(v)\ge \alpha\}$, where $\deg(v)$ denotes the degree of $v$ in $\B(n,n,p)$.

\begin{lemma}\label{lem:keybipartite}
A.a.s.\ $\dim P[U[S_A]]\le 4$ and $\dim P[D[S_B]]\le 4$.
\end{lemma}

{\em Proof of Theorem~\ref{thm:mainbipartite}.\ } The lower bounds in parts (a,b) have been demonstrated in Section~\ref{sec:bipartitleLB}.  
For the upper bound in part (b), by Theorem~\ref{thm:deterministicBipartitie}, a.a.s.\ 
\begin{eqnarray*}
    \dim P_{\B(n,n,p)} &\le& \dim P[U[S_A]]+\dim P[A\cup B\setminus S_A]\\
    &\le&  \dim P[U[S_A]] + \dim P[D[S_B]\cap (A\setminus S_A)] + \dim P[(A\setminus S_A)\cup (B\setminus S_B)]\\
    &\le & \dim P[U[S_A]] + \dim P[D[S_B]] + \dim P[(A\setminus S_A)\cup (B\setminus S_B)],
\end{eqnarray*}
   By Theorem~\ref{thm:Kahn}, $\dim P[(A\setminus S_A)\cup (B\setminus S_B)] = O(c\log^2 c)$. Hence,
    the desired upper bound follows by Lemma~\ref{lem:keybipartite}.\qed

\subsubsection{Proof of Lemma~\ref{lem:keybipartite}}

The proof for the second bound is symmetric to the first one. Thus, we focus on the first assertion of Lemma~\ref{lem:keybipartite}. We call a connected graph {\em bicyclic}, if there are exactly two vertices with degree three, or there is exactly one vertex with degree four, and all the other vertices have degree equal to two. 

\begin{lemma}\label{lem:Properties} 
A.a.s.\ there is no connected bicyclic subgraph $H$ of $\B(n,n,p=c_n/n)$ such that $V(H)\cap A\subseteq S_A$.

   
\end{lemma}

\proof Suppose that $H$ is a connected bicyclic subgraph of $\B(n,n,p)$ with $i$ vertices in $S_A$, $j$ vertices in $B$. Let $t=|E(H)|$. Then, $t=i+j+1$ and moreover, $i-1\le j\le i+1$.  There are $\binom{n}{i}$ ways to choose $V(H)\cap A$, $\binom{n}{j}$ ways to choose $V(H)\cap B$. Then, there are at most $\binom{i+j}{2}$ ways to choose the vertices that have degree greater than two. Finally, 
we claim the following bound on the number of ways to form a bipartite graph on the set of chosen vertices.
\begin{claim}\label{lem:bipartitecount}
There are at most $t^2t^t$ ways to form a bipartite bicyclic graph on the set of chosen vertices. 
\end{claim}
After fixing the edges in $H$, every vertex in $V(H)\cap A$ must be adjacent to at least $\alpha-4$ other vertices in $B$. There are at most $\binom{n}{\alpha-4}^i$ ways to choose these neighbours. It follows that the probability of having a connected bicyclic subgraph $H$ with $i$ vertices in $S_A$ is bounded from above by

\begin{align*}
&\sum_{j=i-1}^{i+1}\binom{n}{i}\binom{n}{j}\binom{i+j}{2}t^2 t^t\binom{n}{\alpha-4}^i p^{t+(\alpha-4)i}\\
&\hspace{1cm} \le \sum_{j=i-1}^{i+1}\frac{ n^{-1}}{(i/e)^i (j/e)^j}(i+j+1)^4(i+j+1)^{i+j+1} \frac{c_n^{i+j+1+(\alpha-4)i}}{((\alpha-4)/e)^{i(\alpha-4)}}\\
&\hspace{1cm} =O( n^{-1} c_n i^5) \sum_{j=i-1}^{i+1}\left(\frac{ec_n}{\alpha-4}\right)^{i(\alpha-4)} \left(\frac{4ec_ni  }{i}\right)^{i} \left(\frac{4eci  }{j}\right)^{j} \quad \text{since $i+j+1\le 4i$}\\
&\hspace{1cm} =O( n^{-1} c_n i^5) \left(\frac{2e}{K}\right)^{i(Kc_n-4)} \left(8ec_n \right)^{2i+1} \quad \text{since $Kc_n-4\ge Kc_n/2$ for large $K$}\\
&\hspace{1cm} =O( n^{-1} c_n i^5) \left(\frac{32e^2}{K}\right)^{i(Kc_n-4)}\quad \text{since $(8ec_n)^{ 2i+1} \le (16 e)^{i(Kc_n-4)} $ for large $K$}.
\end{align*}
Hence, the probability of having a subgraph $H$ described in the lemma is at most
\[
\sum_{i\ge 2} O( n^{-1} c_n i^5) \left(\frac{32e^2}{K}\right)^{i(Kc_n-4)} =o(1),
\]
provided that $K> 32e^2$. 

\noindent{\em Proof of Claim~\ref{lem:bipartitecount}.\ } There are three cases in terms of the degrees of the $i$ vertices in $A$: (a), all of them have degree 2; (b), one of them have degree 3 or 4, and all of the rest have degree 2; (c) two of them have degree 3, and all of the rest have degree 2. 
It suffices to specify the neighbours of each vertex in $i$ which immediately yields an upper bound $(j^2)^i\le j^{2+2i}$ for case (a); an upper bound $(j^2)^{i-1} j^4\le j^{2+2i}$ for case (b), and an upper bound $(j^2)^{i-2} (j^3)^2\le j^{2+2i}$ for case (c). As $t=i+j+1$ and $j\ge i-1$, it follows that $t\ge 2i$. Our claimed bound follows immediately.\qed

\smallskip

\noindent {\em Proof of Lemma~\ref{lem:keybipartite}.\ } Let $L$ be a component of the subgraph $L$ of $\B(n,n,p)$ induced by $U[S_A]$. Suppose $L$ has more than one cycle. Then, there is a connected bicyclic subgraph of $L$ satisfying the condition of Lemma~\ref{lem:Properties}(b). Hence, we may assume that $L$ is either a tree or is unicyclic. Hence, $\dim P[U[S_A]]=\max\{2,\dim P[C]: \text{$C$ is a component of $L$}\}\le 4$ by Theorems~\ref{thm:Trotter} and~\ref{thm:Brightwell}. \qed

\subsection{The co-sparse regime: proof of Theorem~\ref{thm:bipartiteLargep}}

Let $\bar B$ denote the bipartite complement of $B\sim \B(n,n,p)$.  Then, $\bar B\sim \B(n,n,q)$ where $q=1-p$.

\subsubsection{Proof of part (a)}

\begin{claim} Suppose that $q=O(1/n)$ and $q=\omega(1/n^2)$.
A.a.s.\ $\bar B$ contains $\Theta(n^2 q)$ isolated edges, i.e.\ components whose order is exactly two.
\end{claim}

\proof Let $X$ denote the number of isolated edges in $\bar B$. Then,
\[
\ex X=n^2 q (1-q)^{2(n-1)}=\Theta(n^2 q) \quad \text{as $q=O(1/n)$},
\]
and
\[
\ex X(X-1)= n^2(n-1)^2 q^2 (1-q)^{4(n-1)-2} \sim (\ex X)^2.
\]
It follows by Chebyshev's inequality that a.a.s.\ $X\sim \ex X$. \qed \smallskip

We first prove the lower bound on $\dim P_{\B(n,n,p)}$. Let $e_1,\ldots, e_m$ be the set of isolated edges in $\bar B$. Then, a.a.s.\ $m=\Theta(nq^2)$.
Let $(A',B')\subseteq (A,B)$ be the set of vertices in $\cup_{j=1}^m \{e_j\}$. Consider the sub-order $P'$ of $P$ induced by $A'\cup B'$. We observe that $P'$ is a bipartite   
order such that every element in $A'$ is comparable to all but exactly one element in $B'$, and the same holds for every element in $B'$. It is well known that the dimension of $P'$ is equal to $|A'|=m$. Immediately, $\dim P_{\B(n,n,p)}\ge \dim P'=\Theta(n^2 q)$.

Next, we prove the two upper bounds.  First we prove that $\dim P_{\B(n,n,p)} =n- \Omega (n)$. Let $A'\subseteq A$ be the set of elements that are comparable to every element in $B$, and let $B'\subseteq B$ be the set of elements that are comparable to every element in $A$. Let $P'$ be the sub-order of $P$ induced by $(A\setminus A')\cup (B\setminus B')$.

\begin{claim}\label{claim:dimPP'}  $\dim P\le  \dim P' +1$. Moreover,
a.a.s.\ $|A\setminus A'|, |B\setminus B'| \sim n (1-p^{n}) $.
\end{claim}
\proof Let
$d=\dim P'$ and $L_1,\ldots, L_d$ be a realiser of $P'$. Let $L$ be an arbitrary total ordering on $A\cup B$, and $L'$ be the reverse ordering of $L$. For each $1\le j\le d$, extend $L_j$ to $A\cup B$ by $L|_{A'}<L_j|_{(A\setminus A')\cup (B\setminus B')}<L|_{B'}$. Then, add to this set another total ordering defined by $L_1|_{(A\setminus A')} < L'|_{A'}<L'|_{B'}< L_1|_{(B\setminus B')}$. Obviously this set of $d+1$ total orderings is a realiser for $P$.

Let $X=|A\setminus A'|$. Then, $\ex X=n(1-p^{n})$, and $\ex X(X-1)=n(n-1)(1-p^{n})^2 \sim (\ex X)^2$. Thus, by Chebyshev's inequality, a.a.s.\ $X\sim n (1-p^{n})$. The same statement holds for $|B\setminus B'|$ by symmetry.
 \qed \smallskip

By the range of $p$, $n (1-p^{n}) = n-\Omega(n)$ and thus, a.a.s.\ $\dim P=n-\Omega(n)$ by Claim~\ref{claim:dimPP'}. Next, we show that a.a.s.\ $\dim P=O(n^2 q)$, which holds trivially if $q=\Theta(1/n)$.  Suppose that $q=o(1/n)$. Then, by Claim~\ref{claim:dimPP'}, a.a.s.\  $|A\setminus A'|, |B\setminus B'|\sim n-ne^{-qn} \sim  n^2q$, and consequently, a.a.s.\ $\dim P =O(n^2q)$. \qed








\subsubsection{Proof of part (b)} 

Suppose that $q=1-p=\omega(1/n)$ and $q\le n^{-1+1/7}$.  We need the following properties for $\B(n,n,q)$.
\begin{lemma}\lab{lem:Bp} Let $\eps>0$ be fixed. There exists a constant $C>0$ such that
a.a.s.\ $\B(n,n,q)$ has the following properties:
\begin{enumerate}[(P1)]
\item there are $(1+o(1))qn^2$ edges;
\item 
$
d_v\le C(qn+\log n),\quad \text{for all $v\in A\cup B$},
$ where $d_v$ denotes the degree of $v$ in $\B(n,n,q)$.
\item there are no subgraphs isomorphic to $K_{2,3}$;
\item   
$$
\sum_{v\in A\cup B} d_v \ind{d_v\ge (1+\eps)qn} =o(qn^2).
$$
\end{enumerate}
\end{lemma}

Let $B\sim {\B(n,n,p)}$ and $\bar B$ the bipartite complement of $B$. Then $\bar B\sim \B(n,n,q)$.
Suppose that $\R$ is a realiser for $P=P_{\B(n,n,p)}$. The following lemma gives a set of properties of $P$ assuming that $\bar B$ satisfies properties (P1)--(P4) of Lemma~\ref{lem:Bp}.
\begin{lemma}\lab{lem:Realiser} Assume that $\bar B$ satisfies Lemma~\ref{lem:Bp} (P1)--(P4). Let $\bar d_v$ denote the degree of $v$ in $\bar B$.
Let $L\in \R$. Let $y_1$ be the smallest element in $B$ with respect to $L$ and let $x_k<_Lx_{k-1}<_L\cdots<_Lx_1$ be the $k$ greatest elements in $A$ with respect to $L$ such that $x_k<_P y_1$ and all elements $x_{k-1},\ldots, x_1$ are incomparable to $y_1$ in $P$. Then,
 \begin{enumerate}[(a)]
 \item $k\le C(qn+\log n)$;
 \item If $(x,y)\in A\times B$ is incmparable in $P$ and is reversed by $L$ then $x\in \{x_1,\ldots, x_{k-1}\}$;
 \item For each $1\le i\le k-1$, let $R_i$ denote the set of $y\in B$ such that $(x_i,y)$ is incomparable in $P$ and is reversed by $L$. Then, $y||x_j$ for all $y\in R_i$ and all $x_j$ where $j\ge i$; 
 \item $|\cup_{3\le j\le k-1} R_j|\le 1$, $|R_2|\le 2$ if $k\ge 3$, and $|R_1|\le \bar d_{x_1}$ if $k\ge 2$.
 \end{enumerate}
\end{lemma}

We first complete the proof of part (b) of Theorem~\ref{thm:bipartiteLargep} assuming Lemmas~\ref{lem:Bp} and~\ref{lem:Realiser}.
Let $\eps>0$. We show that a.a.s.\ any realiser $|\R|$ must satisfy that $|\R|\ge (1-\eps)n$. Let $\R=\{L_1,\ldots, L_d\}$ be a realiser where $d=\dim P$.  By Lemma~\ref{lem:Realiser}(b,d) we may assume that if some $L\in \R$ reverses more than $(1+\eps/2)qn$ incomparable pairs in $P$ then there is a unique element $x_L\in A$ such that
\begin{itemize}
\item[(a)]  $L$ reverses at least $(1+\eps/2)qn-3$ incomparable pairs $(x_L,y)$ where $y\in B$.
\item[(b)]  $L$ reverses at most 3 other incomparable pairs.
\end{itemize}
If $L\in \R$ reverses more than $(1+\eps/2)qn$ incomparable pairs in $P$, we colour $L\in \R$ red, and colour the corresponding $x_L$ red.
 By Lemma~\ref{lem:Realiser}(d) and (a), every red element in $A$ must have degree at least $(1+\eps/2)qn-3\ge (1+\eps/3)qn$ in $\bar B$. By Property (P4) of Lemma~\ref{lem:Bp} and Lemma~\ref{lem:Realiser}(d), the total number of incomparable pairs reversed by the set of red total orderings in $\R$ is $o(qn^2)$. Hence, by Lemma~\ref{lem:Realiser}(a), a.a.s.\ the number of non-red total orderings in $\R$ must be at least
\[
\frac{(1+o(1))qn^2}{(1+\eps/2)qn}\ge (1-\eps)n,\quad \text{for all sufficiently large $n$},
\]  
which implies that $|\R|\ge  (1-\eps)n$ as desired. It only remains to prove Lemmas~\ref{lem:Bp} and~\ref{lem:Realiser}.
\smallskip

\noindent {\em Proof of Lemma~\ref{lem:Bp}.\ } (P1) and (P2) are standard results in random graph theory following routine concentration arguments, and thus we skip their proofs. Let $X$ denote the number of subgraphs isomorphic to $K_{2,3}$. Then,
\[
\ex X=2 n^{5} q^6 =o(1),\quad \text{since $q\le n^{-1+1/7}$}.
\] 
Hence, (P3) holds a.a.s. Finally, noting that $d_v$ has the binomial distribution ${\bf Bin}(n,q)$, it is immediate that for every $v\in A\cup B$ and for all sufficiently small $\eps>0$,

\begin{eqnarray*}
\ex\left( d_v\ind{d_v\ge (1+\eps)qn} \right) &\le &\sum_{j\ge (1+\eps)qn} j \binom{n}{j} q^j (1-q)^{n-j}\\
&\le & \frac{qn}{1-q}e^{-qn} \sum_{j\ge (1+\eps)qn -1}  \frac{(nq/(1-q))^j}{j!} \\
&=&O\left( qn e^{-qn} \left(\frac{nqe}{(1-q)(1+\eps)qn}\right)^{(1+\eps)qn} \right)\\
&=& qn \exp\left(-\frac{\eps^2}{2}qn+O(q+\eps^3)qn\right)=o(qn). 
\end{eqnarray*}
Therefore,
\[
\ex\left(\sum_{v\in A\cup B} d_v\ind{d_v\ge (1+\eps)qn} \right) =o(qn^2),
\]
and hence (P4) follows by the Markov inequality.\qed
\smallskip

\noindent {\em Proof of Lemma~\ref{lem:Realiser}.\ } Part (a) follows from Lemma~\ref{lem:Bp} (P2). For part (b), notice that for any $x\notin \{x_1,\ldots, x_{k-1}\}$, we must have $x<_L x_{k-1}<_L y_1\le _L y$ and therefore $(x,y)$ cannot be reversed by $L$. For (c), since all elements in $R_i$ are smaller than $x_i$ in $L$, they are smaller than $x_j$ for all $j\le i$ in $L$, which implies that no elements in $R_i$ can be comparable to any of these $x_j$. Part (d) now follows from part (c), as otherwise a copy of $K_{2,3}$ must occur in $\bar B$, which would contradict with Lemma~\ref{lem:Bp}  (P3). \qed
 
 \qed

\section{Non-bipartite POS: proof of Theorem~\ref{thm:nonbipartiteLB}}

In this section, $\G_n\sim \G(n,p)$, $P\sim P_{\G(n,p=c_n/n)}$, and $c>0$ is fixed. We use $<$ to denote the order in natural numbers, and $<_P$ to denote the order in $P$. Throughout this section, $q$ denotes $1-p$.

Suppose that $\xi,\beta$ are given so that the assumptions of the theorem are satisfied.
Partition $[n]$ into four parts $A'\cup A\cup B\cup B'$, where $A'$ contains the $\xi n$ smallest elements (in natural order), $A$ contains the next $n/2-\xi n$ smallest elements, $B$ the next $n/2-\xi n$ smallest elements and $B'$ the greatest $\xi n$ elements. 

Recall $\alpha_c$ from~(\ref{def:alpha_c}). We have the following lemma similar to Lemma~\ref{lem:Bproperties}.

\begin{lemma}\label{lem:incomparable}
 Let $\eps>0$ be fixed; $c\ge 1$ and $\alpha=(1+\eps) \alpha_c$.
\begin{enumerate}[(a)]
\item A.a.s.\ $P$ has $(1-o(1)) \xi^2 n^2$ incomparable pairs $(i,j)\in A'\times B'$;
\item A.a.s.\ for all $S\subseteq A$, $T\subseteq B$ where $|S|, |T|\ge \alpha n$, there is an edge between $S$ and $T$ in $\G(n,p)$.
\end{enumerate}
\end{lemma}
 
 \proof By~\cite[Lemma 3.1 (iii)]{bollobas1997dimension}, a.a.s.\ the maximum size of the down-set or up-set of any element in $P$ is $O(\log n)$. This yields part (a). 
Part (b) follows by a similar proof as in Lemma~\ref{lem:Bproperties} and notice that
\[
\binom{(1/2-\xi)n}{\alpha n}^2(1-p)^{\alpha^2 n^2}\le \binom{n}{\alpha n}^2(1-p)^{\alpha^2 n^2}. \qed
\] 

The following lemma is a generalisation of~\cite[Lemma 3.1]{bollobas1997dimension} with $V=\{1\}$. The proof is almost identical to~\cite[Lemma 3.1]{bollobas1997dimension}, and we include it here for completeness.
\begin{lemma}\label{lem:upset}
Let $U,V$ be a partition of $[n]$, such that for all $(x,y)\in V\times U$, $x<y$ (in the natural order). Then,  for any $s\ge 1$, the probability that $|U_P[V]\cap U|=s$ is equal to 
\[
 q^{|V|(|U|-s)}\prod_{j=0}^{s-1}(1-q^{j+|V|})\prod_{i=1}^s \frac{1-q^{|U|-s+i}}{1-q^i}.
\]
\end{lemma}
\proof Given $S\subseteq U$ where $|S|=s$, let $\E_{S}$ denote the event that $U_P[V]\cap U=S$. Suppose that $t_1<t_2<\cdots<t_{s}$ are the elements in $S$ where $<$ denotes the natural order. Let $k_i=|\{t\in U: t_{i-1}<t<t_i\}|$ for all $2\le i\le s$, and  $k_1=|\{t\in U: t<t_1\}|$, $k_{s+1}=|\{t\in U: t>t_s\}|$.
Then,  with $\kvec$ denoting $(k_1,\ldots,k_{s+1})$ and $\norm{\kvec}$ denoting $\sum_{i=1}^{s+1} k_i = |U|-s$,
\begin{align*}
\pr(|U_P[V]\cap U|=s)&=\sum_{S\subseteq U: |S|=s}\pr(\E_{S})=\sum_{\kvec} \prod_{j=0}^{s-1}(1-q^{j+|V|}) \prod_{j=1}^{s} q^{k_j(|V|+j-1)}\\ 
&=\prod_{j=0}^{s-1}(1-q^{j+|V|}) \cdot q^{(|V|-1)\norm{\kvec}}  \sum_{\kvec} \prod_{i=1}^{s} q^{ik_i}\\ 
&= \prod_{j=0}^{s-1}(1-q^{j+|V|}) \cdot q^{(|V|-1)\norm{\kvec}} \cdot [z^{|U|-s}] \prod_{i=1}^{s} \frac{1}{1-q^iz}\\
&= \prod_{j=0}^{s-1}(1-q^{j+|V|}) \cdot q^{(|V|-1)\norm{\kvec}} q^{|U|-s} \prod_{i=1}^{|U|-s} \frac{1-q^{s+i}}{1-q^i}\\
&= \prod_{j=0}^{s-1}(1-q^{j+|V|}) \cdot q^{|V| (|U|-s)} \prod_{i=1}^{s} \frac{1-q^{|U|-s+i}}{1-q^i}. \qed
\end{align*}
Let $U,V$ be as in Lemma~\ref{lem:upset}. By considering the dual of $P$ it follows immediately that 
\begin{equation}
\pr(|D_P[U]\cap V|=s)=\prod_{j=0}^{s-1}(1-q^{j+|U|}) \cdot q^{|U| (|V|-s)} \prod_{i=1}^{s} \frac{1-q^{|V|-s+i}}{1-q^i}. \label{eq:downset}
\end{equation}

\begin{lemma} \label{lem:expansion} There exists $\eps=\eps(c)>0$ such that 
a.a.s.\ the following is true for any $S'\subseteq A'$ and $T'\subseteq B'$ with $|S'|=|T'|=\beta n$:
\begin{enumerate}[(a)]
\item $|U[S']\cap A| \ge tn$; 
\item $|D[T']\cap B| \ge tn$,
\end{enumerate}
where $t=(1+\eps)\alpha_c$.
\end{lemma}

\proof We only need to prove part (a), as part (b) follows from (a) by consider the dual of $P$.  By Lemma~\ref{lem:upset} (to the sub-poset induced by $A'\cup A$), 
\[
\pr(|U[S']\cap A|=tn)= q^{\beta n (|A|-tn)} \prod_{j=0}^{tn-1}(1-q^{j+\beta n}) \prod_{i=1}^{tn} \frac{1-q^{|A|-tn+i}}{1-q^i}.
\]
First, we express the right hand side above into form
\[
\exp((1+o(1))ng(c,\xi,t,\beta))\quad \text{for some function $g$}.
\]
Notice that $|A|=(1/2-\xi)n$. Thus,
\begin{align*}
q^{\beta n (|A|-tn)} &= \left(1-\frac{c}{n}\right)^{\beta n (|A|-tn)} = \exp\left(-n \left(c\beta (1/2-\xi-t)  +O\left( \frac{c^2(1/2-\xi-t)}{n}\right) \right) \right).
\end{align*}
Next, we estimate $\prod_{i=1}^{tn} \frac{1}{1-q^i} $.
\begin{align*}
\prod_{i=1}^{tn} \frac{1}{1-q^i} &= \exp\left(-\sum_{i=1}^{tn} \log(1-(1-c/n)^i) \right)\\
&=\exp\left(-\sum_{i=1}^{tn} \log(1-e^{ci/n+O(c^2 i/n^2)}) \right)\\
&=\exp\left(-n (1+o_n(1)) \int_{0}^{t} \log(1-e^{-cx}) dx \right),
\end{align*}
where $o_n(1)$ is a term vanishing to 0 as $n\to\infty$. With a change of variable $y=e^{-cx}$ we obtain 
\begin{align*}
\int_{0}^{t} \log(1-e^{cx}) dx &= \frac{1}{c} \int_{e^{-ct}}^1 y^{-1}\log(1-y) dy\\
&=-\frac{1}{c}(Li_2(1)-Li_2(e^{-ct})),
\end{align*}
where $Li_2(z):=-\int_{0}^z y\log(1-y) dy$ is the dilogarithm function, which has the following expansion for all $z\in \mathbb{C}$ where $|z|\le 1$:
\begin{equation}
Li_2(z)=\sum_{k=1}^{\infty} \frac{z^k}{k^2}. \label{eq:expansion}
\end{equation} 
Hence,
\begin{align}
n^{-1}\log\left(\prod_{i=1}^{tn} \frac{1}{1-q^i}\right) =(1+o_n(1))\frac{1}{c}\left(Li_2(1)-Li_2(e^{-ct}) \right).\label{eq1}
\end{align}
Similarly, we obtain that
\begin{align}
n^{-1}\log\left(\prod_{j=0}^{tn-1}(1-q^{j+\beta n})\right)&=(1+o_n(1)) \int_{0}^{t} \log(1-e^{-c\beta-cx}) dx \nonumber\\
&=(1+o_n(1)) \int_{e^{-c\beta-ct}}^{e^{-c\beta}} \frac{1}{c} y^{-1} \log(1-y) dy  \nonumber\\
&= (1+o_n(1))(Li_2(e^{-c\beta-ct})-Li_2(e^{-c\beta})), \label{eq2}
\end{align}
and
\begin{align}
n^{-1}\log\left( \prod_{i=1}^{tn} (1-q^{|A|-tn+i})\right)=(1+o_n(1))(Li_2(e^{-c(1/2-\xi-t)-ct})-Li_2(e^{-c(1/2-\xi-t)})).\label{eq3}
\end{align}
Finally, the number of ways to choose $S'$ in $A'$ is 
\begin{equation}
\binom{\xi n}{\beta n} \le\left(\frac{e\xi}{\beta}\right)^{\beta n}=\exp\left(n\cdot \beta \log(e\xi/\beta)\right).\label{eq4}
\end{equation}
Let $\B$ denote the event that there exists $S'\subseteq A'$ where $|S'|=\beta n$ and $|U[S']\cap A|=tn$. Combining~\eqn{eq1}--\eqn{eq4} and recalling $f$ in~\eqn{def:f} we 
obtain that
\begin{align*}
n^{-1}\log\,\pr(\B) &= (1+o_n(1)) f(c,\xi,\beta,t).
\end{align*}
Since $\sup_{t\le \alpha_c}f(c,\xi,\beta,t)<0$, by continuity of $f$ there exists $\eps=\eps(c)>0$ such that $$\sup_{t\le (1+\eps)\alpha_c}f(c,\xi,\beta,t)<0,$$ which implies that $\pr(\B)=o(1)$. This completes the proof of Lemma~\ref{lem:expansion}.
\qed \smallskip

\noindent{\em Proof of Theorem~\ref{thm:nonbipartiteLB}.\ } Part (a) follows by a similar argument as in the lower bound proof of Theorem~\ref{thm:mainbipartite}.
By Lemma~\ref{lem:incomparable}(a), the number of incomparable pairs $(i,j)\in A'\times B'$ is a.a.s.\ asymptotic to $\xi^2 n^2$.
By Lemma~\ref{lem:incomparable}(b) and Lemma~\ref{lem:expansion}, a.a.s.\ every linear extension of $P$ can reverse at most $2\beta\xi n^2$ pairs of incomparable pairs of elements in $A'\times B'$. Hence, the size of any realiser must be at least
\[
\frac{\xi^2 n^2}{2\beta\xi n^2} = \frac{\xi}{2\beta}.
\]

For part (b),  set
\[
\beta=\exp(-c/6+c/\log c),\quad \xi=1/\log c.
\]  
It follows that 
\begin{equation}
c\beta\to 0, \quad 1/2-\xi-t\ge 1/4\ \text{for all $t\le 3\log c/c$}, \quad\text{as $c\to \infty$.} \label{eq:large-c}
\end{equation}
Using the expansion~\eqn{eq:expansion}, we obtain that
\[
Li_2(1)- Li_2(e^{-ct}+Li_2(e^{-c\beta-ct})-Li_2(e^{-c\beta}) + Li_2(e^{-c(1/2-\xi-t)-ct})-Li_2(e^{-c(1/2-\xi-t)}) )=\sum_{k=1}^{\infty} \frac{\rho_k}{k^2},
\]
where, by~\eqn{eq:large-c}, for all $k\ge 1$,
\begin{align*}
\rho_k&=(1-e^{-c\beta k})(1+e^{-ctk}) - e^{-c(1/2-\xi-t)k}(1-e^{-ctk})\le (1+e^{-ctk})(1-e^{-c\beta k}) \le 2 c\beta k.
\end{align*}
Thus,
\begin{align}
\sum_{k=1}^{\infty} \frac{\rho_k}{k^2} &\le 2 \left(\sum_{k=1}^{1/c\beta} \frac{c\beta k}{k^2} + \sum_{k=1/c\beta}^{\infty}\frac{1}{k^2}\right) =2 \Big( c\beta ( \log(1/c\beta)+O(1)) +O(c\beta)\Big)\nonumber\\
&=2 c\beta \Big( \log(1/c\beta) +O(1)\Big). \label{eq:dilogarithmApprox}
\end{align}
We show that for all sufficiently large $c$, 
$$\sup_{t\le 3\log c/c}f(c,\xi,t,\beta)<0.
$$
By~\eqn{eq:dilogarithmApprox}, $f(c,\xi,t,\beta)=g(c,\xi,t,\beta)+O(\beta)$ where 
$$g(c,\xi,t,\beta)=\beta\log(e\xi/\beta)-c\beta(1/2-\xi-t) + 2\beta \log(1/c\beta),
$$ which is an increasing function of $t$. 
Since
\[
g(c,\xi,3\log c/c,\beta)= \beta \left(-\frac{2c}{\log c} +O(\log c)\right),
\]
it follows that for all sufficiently large $c$,
\[
\sup_{t\le 3\log c/c} f(c,\xi,t,\beta)<0.
\]
By Lemma~\ref{lem:asymptotic}, $\alpha_c<3\log c/c$ for all sufficiently large $c$.  Thus, $\xi+\alpha_c<1/2$ and $\sup_{t\le \alpha_c}f(c,\xi,t,\beta)<0$ are both satisfied. By part (a), a.a.s.\ $\dim P_{\G(n,p)}\ge \frac{1}{2\log c} \exp(c/6-c/\log c) \ge \exp(\gamma c) $, where $\gamma=\frac{1}{6}-1/\log c - \log(2\log c)/c$. This completes the proof for part (b). \qed


\section{Non-bipartite POS: proof of Theorem~\ref{thm:mainnonbipartite}}



Throughout this section, $p=c_n/n$ where $1\le c_n\ll \log n$ and $q=1-p$.
Set 
\begin{equation}
\alpha=Kc_ne^{c_n},\quad S=\{x: |D[x]\cup U[x]|\ge 10Kc_n\cdot\alpha\}.\label{def:S}
\end{equation}
 In light of Theorem~\ref{thm:deterministicNonBipartitie}, it suffices to bound the dimension of the sub-poset induced by $D[S]$ (the bound on $\dim P[U[S]]$ follows by symmetry), as established in the following key lemma.
\begin{lemma}\label{lem:nonbipartiteKey}
A.a.s.\ every component of the cover graph of $P[D[S]]$ is either a tree or unicyclic. 
\end{lemma}

\noindent{\em Proof of Theorem~\ref{thm:mainnonbipartite}.\ } 
By considering the dual of $P$ and Lemma~\ref{lem:nonbipartiteKey}, a.a.s.\ every component of the cover graph of $P[U[S]]$ is either a tree or unicyclic.
By Theorems~\ref{thm:Trotter} and~\ref{thm:Brightwell}, a.a.s.\ $\dim P[D[S]], \dim P[U[S]]\le 4$. Then, Theorem~\ref{thm:deterministicNonBipartitie} implies that a.a.s.\ $\dim P\le 2(8+\dim P[U\cup D]+\dim P[[n]\setminus S])$, where $U=U[S]\setminus S$ and $D=D[S]\setminus S$. By the definition of $S$, in both $P[U\cup D]$ and  $P[[n]\setminus S]$, every element  is comparable to at most $10Kc_n\alpha=O(c_n^2e^{c_n})$ other elements. Thus, by Theorem~\ref{thm:Kahn}, $\dim P[U\cup D], \dim P[[n]\setminus S]=O(c_n^4e^{c_n})$. \qed  \smallskip

It only remains to prove Lemma~\ref{lem:nonbipartiteKey}.
 Unlike $P_{\B(n,n,p)}$ where the comparable graph and the cover graph are both $\B(n,n,p)$, the cover graph of $P_{\G(n,p)}$ is typically a proper subgraph of $\G(n,p)$, while the comparable graph of $P_{\G(n,p)}$ is a super graph of $\G(n,p)$ with a much greater density. To study $P_{\G(n,p)}$, it is much more convenient to consider the oriented version of $\G(n,p)$ where every edge $ij$ in $\G(n,p)$ is oriented from $i$ to $j$ if $i<j$. Since every graph on $[n]$ has a unique orientation we keep the same notation $\G(n,p)$ for its oriented version. Note that two elements $i$ and $j$ in $P_{\G(n,p)}$ are comparable if and only if there is a directed path joining $i$ and $j$ in $\G(n,p)$.
 
The key idea in the proof of upper-bounding $\dim P_{\B(n,n,p)}$ is that $S_A$ is much smaller than $A$, and thus, it is not likely to join vertices in $S_A$ by vertices in $B$ into a connected bicyclic subgraph. Situations in $P_{\G(n,p)}$ are quite different, as elements with large down-sets or up-sets form clusters. Suppose $v$ is an element with a large down-set. Then all the elements $u\in U_P(v)$  have a large down-set. Therefore, we expect many large cliques in the comparable graph of $P_{\G(n,p)}$, and they intersect in non-trivial ways. However, this does not mean that the cover graph of $P_{\G(n,p)}$ induced by the elements with large down-sets have complicated graph structures. Indeed, Lemma~\ref{lem:nonbipartiteKey} shows that every component of the cover graph has at most one cycle. Unlike the study of $\dim P_{\B(n,n,p)}$, analysing the cover graph of  $P[D[S]]$ for $P\sim P_{\G(n,p)}$ necessitates the identification and careful examination of suitable chosen structures. 


Let $G_D[S]$ denote the cover graph of $P[D[S]]$. A helpful observation is the following.
\begin{obs}\label{obs:subgraph}
Let $V$ denote the set of elements in $P[D[S]]$.  Then, $G_D[S]\subseteq \G(V,p)$.
\end{obs}

{\em Proof of Observation~\ref{obs:subgraph}.\ } Suppose on the contrary that there exists an arc $(x,y)$ in $G_D[S]\setminus \G(V,p)$. It implies that there is a directed $x,y$-path $P_{xy}$ in $\G(n,p)$ such that $P_{xy}$ has length at least two, and none of the internal vertices is in $V$.  On the other hand, there is a directed path $P_{yz}$ from $y$ to some $z\in S$ since $y$ is in the down-set of $S$. By considering the directed path obtained from concatenating $P_{xy}$ and $P_{yz}$, it follows that all the internal vertices on $P_{xy}$ must be in the down-set of $z$, and thus must be in $V$, leading to a contradiciton. \qed\smallskip

It is easy to see that the cover graph of $P[D[S]\cup U[S]]$ does not have the property in Observation~\ref{obs:subgraph}, which is one of the main reasons that we complete the proof of Theorem~\ref{thm:mainnonbipartite} by studying the dimension of $P[D[S]]$ and $P[U[S]]$, instead of studying $\dim P[D[S]\cup U[S]]$.

 Suppose that $G_D[S]$ has a component with at least two cycles. Then, $G_D[S]$ has a connected bicyclic subgraph $M$. Let $P_M$ denote the partial order with cover graph $M$  (notice that $P_M$ is not necessarily an induced sub-order of $P$).
 Let $\overline{M}$ denote the set of maximal elements and $\underline{M}$ denote the set of minimal elements in $P_M$. That is, for all $u\in \overline{M}$, and for all $z\in M$, either $z<_P u$ or $z||u$.  Colour the vertices in $M$ red if their degrees are greater than two; notice that there can be at most two red vertices since $M$ is bicyclic. For all the vertices that are not in $\overline{M}\cup\underline{M}$ that were not coloured red, colour them blue; notice that all blue vertices must have degree two in $M$. See an example of $M$ on the left of Figure~\ref{f:example1}. In this figure, elements higher in position are greater in the natural order. There are eight elements in $\overline{M}$, seven elements in $\underline{M}$; two red elements, none of them are in $\overline{M}\cup\underline{M}$ in this example.
 \begin{figure}[h]
 \begin{center}
 \includegraphics[scale=0.6]{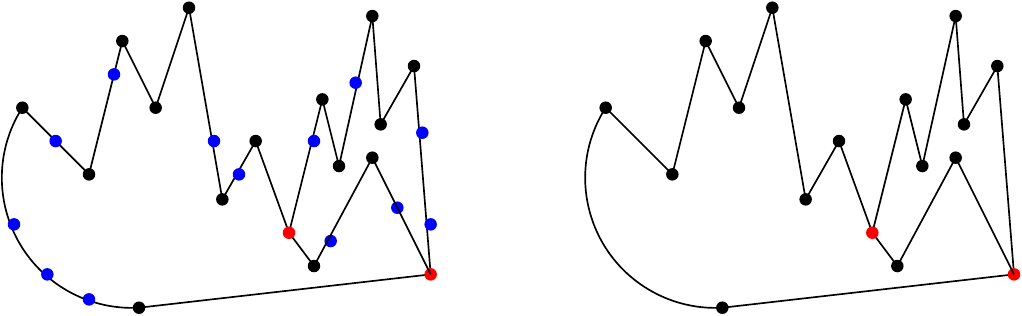}
 \end{center}
 \caption{$M$ and $W(M)$}
 \label{f:example1}
 \end{figure}
 
 Next, we define an auxiliary arc-weighted (multi)digraph $W(M)$ associated to $M$. 
The vertices of $W(M)$ is the union of $\overline{M}\cup \underline{M}$ and the red vertices in $M$. There is an arc $(x,y)$ in $W(M)$  if there is a directed path from $x$ to $y$ in $M$ whose internal vertices are all blue, and the weight of $(x,y)$ is equal to the length of that directed path. If there are two directed paths (with blue internal vertices) in $M$ from $x$ to $y$, then we keep both copies of $(x,y)$ in $W(M)$ representing these two directed paths. Thus, $W(M)$ may have one double arc.  The figure on the right of Figure~\ref{f:example1} is $W(M)$ for $M$ given on the left. The direction of the arcs are omitted in the figure, as obviously all the arcs go from the lower element to the higher. Obviously by construction $W(M)$ is connected and bicyclic. Moreover, all but at most two vertices in $W(M)$ are in $\overline{M}\cup \underline{M}$, and the vertices in $W(M)$ that are not  in $\overline{M}\cup \underline{M}$ must be red.  

\begin{obs}\label{ob:bipartitelike}
\begin{enumerate}[(a)]
\item Every arc $(x,y)$ in $W(M)$ joins some $x\in {\underline M}$ and some $y\in \overline{M}$, unless $x$ or $y$ is red.
\item There is a directed path $P_x$ in $G_D[S]$ from $x$ to some vertex in $S$, for every $x\in \overline{M}$.
\end{enumerate}
\end{obs}
The first observation above follows from the definition of $W(M)$, and the second observation follows from the fact that  $M$ is contained in the down-set of $S$.

To simplify the analysis, we hope to identify a set of directed paths as in Observation~\ref{ob:bipartitelike}(b) that ``behave somewhat independently''. 
We say the triple $(x,y,z)$ in $W(M)$ is a {\em \trilink} if 
\begin{itemize}
\item none of $x,y,z$ or their neighbours in $W(M)$ are red; and
\item $(x,y)$ and $(z,y)$ are arcs in $W(M)$.
\end{itemize}
The figures in Figure~\ref{f:example2} below are the example of $(M,W(M))$ given previously. We marked two {\trilink}s in $W(M)$, each contained in a green dotted curve.  
Note from the definition that if $(x,y,z)$ is a \trilink\ then necessarily $x,z\in \underline{M}$ and $y\in\overline{M}$.
 \begin{figure}[h]
 \begin{center}
 \includegraphics[scale=0.6]{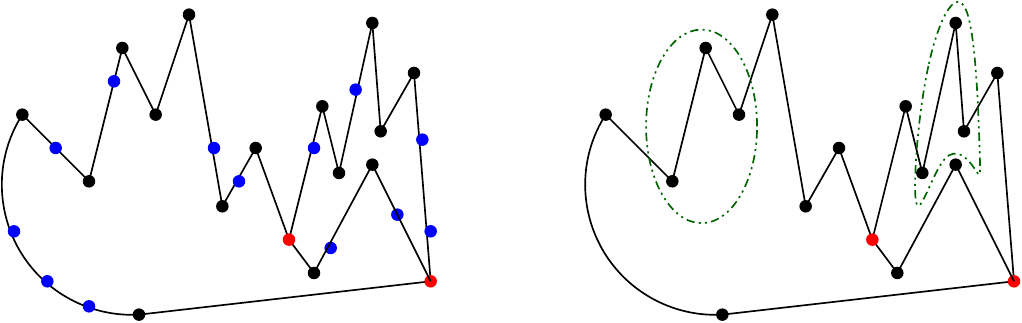}
 \end{center}
 \caption{{\trilink}s}
 \label{f:example2}
 \end{figure}

We say that two {\trilink}s $(x_1,y_1,z_1)$ and $(x_2,y_2,z_2)$ are {\em well-separated} if the graph distance in $W(M)$  between $\{x_1,y_1,z_1\}$ and $\{x_2,y_2,z_2\}$ is at least four. In the example we gave in Figure~\ref{f:example2}, the graph distance between the two {\trilink}s is six, and so they are well-separated.


\begin{lemma}\label{lem:MTP}
Suppose that $G_D[S]$ has a component $C$ with at least two cycles. Then, $G_D[S]$ has a connected bicyclic subgraph $M$, together with a collection ${\mathcal T}$ of pairwise well-separated {\trilink}s in $W(M)$ and a collection ${\mathcal P}$ of directed paths that satisfy the following properties:
\begin{itemize}
\item[(P1)] $|{\mathcal T}|\ge t/21-2$, where $t$ is the number of edges in $W(M)$;
\item[(P2)] For each \trilink\ $(x,y,z)\in {\mathcal T}$ there is a directed path $P_y\in {\mathcal P}$ in $G_D(S)$ from $y$ to some $s_y\in S$ and vice versa;
\item[(P3)] All the directed paths in ${\mathcal P}$ are vertex disjoint;
\item[(P4)] $D_P(s_y)$, $U_P(s_y)$, $D_P(s_{y'})$, and $U_P(s_{y'})$ are pairwise disjoint for every pair of distinct {\trilink}s $(x,y,z),(x',y',z')\in {\mathcal T}$. 
\end{itemize}
\end{lemma}

\proof Choose $M$ such that the total number of maximal elements is minimized. Let ${\mathcal T}$ be a maximum collection of pairwise well-separated {\trilink}s in $W(M)$, and for each \trilink\ $(x,y,z)\in {\mathcal T}$, let $P_y$ be an arbitrary directed path from $y$ to some $s_y\in S$, whose existence is guaranteed by Observation~\ref{ob:bipartitelike}(b). Let ${\mathcal P}$ be the collection of such directed paths. We demonstrate that our choice of $(M,{\mathcal T}, {\mathcal P})$ satisfies (P1)--(P4).
 
 (P1): it holds trivially when $t\le 42$ by taking ${\mathcal T}=\emptyset$. Suppose that $t\ge 43$. Let $M'$ be the vertex-induced subgraph of $W(M)$ obtained by deleting the red vertices in $W(M)$ and their neighbours. Since there were either two  red vertices each of degree three, or a single red vertex with degree 4 in $W(M)$, it is easy to see that $M'$ has at least $t-12$ edges. Moreover, $M'$ is a collection of at most three disjoint paths (not necessarily directed paths), one of which --- call it $P$ --- must have length at least $(t-12)/3$. Since every edge in $M'$ joins a vertex in $\overline{M}$ and a vertex in $\underline{M}$, there must exists at least $\lfloor((t-12)/3)/7\rfloor \ge t/21-2$ well-separated {\trilink}s by taking a maximum collection of pairwise well-separated {\trilink}s on $P$. 

 (P2) holds trivially by the definition of ${\mathcal P}$.

(P3) and (P4) hold due to the minimality of the number of maximal elements in the choice of $M$. It is straightforward to check that if (P3) or (P4) were violated then one can find a connected bicyclic subgraph with fewer maximal elements than those in $M$. Although several cases must be considered, each can be checked with ease. We verify two representative cases below and omit the remaining ones, whose verification is  straightforward. 

Suppose $(x_1,y_1,z_2),(x_2,y_2,z_2)\in {\mathcal T}$ are on the same path $P$ in $M'$, and $P_{y_1}$ and $P_{y_2}$ intersect at some vertex $w$. Let $P'$ denote the path joining $y_1$ and $y_2$ in $M$ who corresponds to the sub-path of $P$ joining $y_1$ and $y_2$ in $M'$. Let $P_1$ and $P_2$ be the sub-paths of $P_{y_1}$ and $P_{y_2}$ joining $y_1$ and $w$, and joining $y_2$ and $w$ respectively. Then, $M-P'+P_1+P_2$ is a connected bicyclic subgraph of $G_D[S]$ that has fewer maximal elements than $M$.

As a second case, suppose $(x_1,y_1,z_2),(x_2,y_2,z_2)\in {\mathcal T}$ are on the same path $P$ in $M'$ and $P_{y_1}$ and $P_{y_2}$ are vertex-disjoint directed paths from $y_1$ to $s_{y_1}$, and from $y_2$ to $s_{y_2}$. Suppose on the contrary that there exists $w\in D[s_{y_1}]\cap D[s_{y_2}]$. Then there is a directed $w,s_{y_1}$-path $P_1$, and a directed $w,s_{y_2}$-path $P_2$. Let $P'$ be defined as in the previous case. Suppose that $w$ is not in $M$. Then, $M-P'+P_1+P_2$ is a connected bicyclic subgraph of $G_D[S]$ that has fewer maximal elements than $M$. The case that $w$ is in $M$, as well as the case where $(x_1,y_1,z_2),(x_2,y_2,z_2)\in {\mathcal T}$ are not on the same path in $M'$ can be checked in a similar manner.
\qed




We will show that structures $(M, {\mathcal T}, {\mathcal P})$ as described in Lemma~\ref{lem:MTP} are unlikely to appear as subgraphs of $G_D[S]$. We need a few probabilistic tools as follows.
\begin{lemma}\label{lem:tools}
Let $U\subseteq [n]$ and let $G_{\bar U}$ be the subgraph of $\G(n,p)$ induced by $[n]\setminus U$. Let $P_{\bar U}$ be the poset with cover graph $G_{\bar U}$.
\begin{enumerate}[(a)]
\item Let $u\in [n]\setminus U$. Then, $\pr(|D_{P_{\bar U}}(u)|\ge k), \pr(|U_{P_{\bar U}}(u)|\ge k)\le (1-q^n)^k$;
\item Let $u,v\in [n]\setminus U$. Let ${\mathcal P}_{k}$ be the collection of tuples $(u, x_1,x_2,\ldots,x_{k-1},v)$ where $u<x_1<x_2<\ldots<x_{k-1}<v$, and $x_1,\ldots,x_{k-1}\in [n]\setminus U$. For each $P=(u, x_1,x_2,\ldots,x_{k-1},v)\in {\mathcal P}_t$, let ${\mathcal E}_P$ denote the probability that $P$ is a directed path in $\G(n,p)$. Then, $\sum_{P\in {\mathcal P}_{k}} \pr({\mathcal E}_P) \le kc_n^k/nk!$.
\end{enumerate}
\end{lemma}

{\em Proof of Lemma~\ref{lem:tools}.\ } For (a), obviously, $\pr(|D_{P_{\bar U}}(u)|\ge k)\le \pr(|D_{P}(u)|\ge k)$, and we can obtain the bound in part (a) easily from~\eqn{eq:downset}. However, it follows even more straightforwardly from~\cite[Lemma 3.1 (iii)]{bollobas1997dimension}.

For (b), 
there are $\binom{v-u-1}{k-1}\le kn^{k-1}/k!$ ways to choose $ x_1,x_2,\ldots,x_{k-1}$, and thus 
\[
\sum_{P\in {\mathcal P}_{k}} \pr({\mathcal E}_P) \le\frac{kn^{k-1}}{k!}(c_n/n)^k = \frac{k}{n}\frac{c_n^k}{k!}.\qed
\]

\begin{cor}\label{cor:tools} Let $U, {\mathcal P}_k$ and ${\mathcal E}_P$ be defined as in Lemma~\ref{lem:tools}. Then,
\begin{enumerate}[(a)]
\item For every $u\in[n]$, $\pr(|D_P[u]\cup U_P[u]|\ge \alpha)\le (1+o(1))\exp(-Kc_n)$.
\item $\sum_{k\ge 1} \sum_{P\in {\mathcal P}_k} \pr({\mathcal E}_P)\le c_ne^{c_n}n^{-1}$ and $\sum_{k\ge 2Kc_n} \sum_{P\in {\mathcal P}_k} \pr({\mathcal E}_P)\le n^{-1}\exp(-Kc_n)$.
\end{enumerate}
\end{cor}

\proof Part (a) follows immediately from Lemma~\ref{lem:tools}(a) (with $U=\emptyset$). Part (b) follows from Lemma~\ref{lem:tools}(b) by summing over all $k\ge 1$ and over $k\ge 2Kc_n$. \qed\smallskip

\noindent{\em Proof of Lemma~\ref{lem:nonbipartiteKey}.\ } Suppose $G_D[S]$ has a connected bicyclic subgraph $M$ and a collection ${\mathcal P}$ of dipaths satisfying all the properties in Lemma~\ref{lem:MTP}. By the claim below, we may assume that $W(M)$ has at least 100 vertices.
\begin{claim} \label{claim:smallM}
Let $B>0$ be a fixed integer. A.a.s.\ there are no connected bicyclic subgraph $M$ in $G_D[S]$ such that $W(M)$ has at most $B$ vertices.
\end{claim}

We first bound the number of ways to construct $W(M)$, and then $M$.  We define a family ${\mathcal F}$ of connected directed bicyclic graphs (who can possibly contain a double edge) that can be candidates for $W(M)$ as follows. The vertices are marked $+$, or $-$, or coloured red, so that they respond to the vertices in $\overline{M}$, $\underline{M}$, and the red vertices in $W(M)$.
\begin{itemize}
\item The vertices whose degree are greater than two are coloured red;
\item Each vertex can be marked $+$ or $-$, or unmarked;
\item Every vertex except for the red vertices must be marked;
\item Every arc $(x,y)$ joins some $x$ marked $-$ and some $y$ marked $+$, unless $x$ or $y$ is red.
\end{itemize}

Next, we bound the number of graphs in ${\mathcal F}$ on a specified set of vertices.

\begin{claim}\label{claim:count}
Let $\Sigma_1,\Sigma_2$ be disjoint sets. Let $R$ be a set (not necessarily disjoint from $\Sigma_1\cup\Sigma_2$) of cardinality between one and two,  and let $a=|\Sigma_1\cup\Sigma_2\cup R|$.
The number of graphs in ${\mathcal F}$ on $\Sigma_1\cup\Sigma_2\cup R$ such that all vertices in $\Sigma_1$ are marked $+$, all vertices in $\Sigma_2$ are marked $-$, and all vertices in $R$ are coloured red is at most $O(a^{4a})$. 
\end{claim}
Given $a,h,\ell$, there are $\binom{n}{h,\ell}$ ways to choose disjoint subsets $\Sigma_1$ and $\Sigma_2$ in $[n]$ such that  $|\Sigma_1|=h$, $|\Sigma_2|=\ell$ and at most $n^2$ ways to choose a set  $R$ of one or two red vertices so that the total number of vertices in $\Sigma_1\cup\Sigma_2\cup R$ is equal to $a$. By claim~\ref{claim:count}, given these choices, there are at most $O(a^{4a})$ ways to form a graph $H$ in ${\mathcal F}$ on these chosen vertices. Next, for each such $H$, we bound the probability that $G_D(S)$ has a connected bicyclic subgraph $M$ such that $W(M)$ is $H$. Let $t$ denote the number of arcs in $H$.
 Notice that $h+\ell\le a\le h+\ell+2$ and $t=a+1$ since $H$ is connected and bicyclic. Let $e_1,\ldots, e_t$ be an enumeration of all the arcs in $H$.  
 \begin{claim} \label{claim:M}
 Let $\Sigma_1,\Sigma_2,R$ and $H$ on $\Sigma_1\cup\Sigma_2\cup R$ be given.  
Let $\wvec=(w_j)_{j=1}^t$ be positive integers. Let ${\mathcal M}$ denote the set of connected bicyclic graphs on $[n]$ such that for each $M\in {\mathcal M}$, $W(M)=H$, and
\begin{itemize}
\item 
for each pair $(x,y)$, if $(x,y)$ is an arc in $H$ then there is a directed path $P_{xy}$ from $x$ to $y$ in $M$;
\item all the directed paths $P_{xy}$ mentioned above are internally disjoint;
\item the length of $P_{xy}$ is equal to $w_j$ if $(x,y)=e_j$.
\end{itemize}
Let ${\mathcal E}_M$ denote the probability that $M$ is a subgraph of $\G(n,p)$. 
Then, $\sum_{M\in {\mathcal M}}\pr({\mathcal E}_M)\le \prod_{j=1}^t \frac{w_j}{n} c_n^{w_j}/w_j!$.
\end{claim}

Let $M \in {\mathcal M}$, and suppose that $W(M)$ has $a\ge 100$ vertices. Let $\wvec$ be the resulting edge weights of $W(M)$. Let ${\mathcal T}$ be a collection of $k:=\lceil a/21-2\rceil$ well-separated {\trilink}s, whose existence is guaranteed by Lemma~\ref{lem:MTP}. There are at most $2^{a}$ ways to specify ${\mathcal T}$. For each {\trilink} $(x,y,z)\in {\mathcal T}$, we specify some $s_y\in [n]$ (it is possible that $s_y=y$). There are $n^{k}$ ways to choose these vertices $s_y$ for each {\trilink} so that they are pairwise distinct. 
Take an arbitrary enumeration of the {\trilink}s $(x_i,y_i,z_i)_{1\le i\le k}$. We expose the downsets $D_P(s_{y_i})$, $U_P(s_{y_i})$ in increasing order of $i$. Recall that $(H,\wvec)=W(M)$. 

\begin{claim}\label{claim:largeDownset} For every $1\le i\le k$, the following holds.
\begin{enumerate}[(a)]
\item Let $A_i$ be the set of arcs incident to $y_i$ in $H$, and let $w^{(i)}=\sum_{j} w_j 1_{\{e_j\in A_i\}}$ be the total weight of the arcs in $A_i$. Then, either $w^{(i)}\ge 4Kc_n$ or $|D_P[s_{y_i}]\setminus M|\ge 4Kc_n \alpha$, or $|U_P[s_{y_i}]\setminus M|\ge 4Kc_n\alpha$.
\item Let $\{D_P(s_{y_j}), U_P(s_{y_j}):\, j<i\}$ be given and suppose that $|D_P[y_i]\cap M|<4Kc_n$.  Let ${\mathcal D}_i$ be the collection of subsets of $[n]$ disjoint from $\{D_P[s_{y_j}], U_P(s_{y_j}):\, j<i\}$ such that $|D\setminus M|\ge 4Kc_n \alpha $ and $y_i\in D$ for each $D\in {\mathcal D}_i$. Let ${\mathcal E}_D$ denote the event that  $D_P[s_{y_i}]=D$ for each $D\in {\mathcal D}_i$. Then,
$\sum_{D\in{\mathcal D}_i}\pr({\mathcal E}_D \mid M, \{D_P[s_{y_j}], U_P(s_{y_j}):\, j<i\}) \le n^{-1}\exp(-Kc_n/2)$.
\item Let $\{D_P(s_{y_j}), U_P(s_{y_j}):\, j<i\}$ be given and suppose that $|D_P[y_i]\cap M|<4Kc_n$.  Let ${\mathcal U}_i$ be the collection of subsets of $[n]$ disjoint from $\{D_P[s_{y_j}], U_P(s_{y_j}):\, j<i\}$ such that $|U\setminus M|\ge 4Kc_n \alpha $ and $y_i\in U$ for each $U\in {\mathcal U}_i$. Let ${\mathcal E}_U$ denote the event that  $U_P(s_{y_i})=U$ for each $U\in {\mathcal U}_i$. Then,
$\sum_{U\in{\mathcal U}_i}\pr({\mathcal E}_U \mid M, \{D_P[s_{y_j}], U_P(s_{y_j}):\, j<i\}) \le n^{-1}\exp(-Kc_n/2)$.
\end{enumerate}
\end{claim}
Let ${\mathcal B}$ denote the event that $G_D[S]$ has a connected bicyclic component.
By all the claim above, $\pr({\mathcal B})$ is at most
\begin{align}
\sum_{h,\ell,a} \sum_{(\Sigma_1,\Sigma_2,R)} \sum_{\wvec} \sum_{M}   \pr({\mathcal E}_M)\sum_{{\mathcal T}}\sum_{s_{y_j}: j\in {\mathcal N}} \sum_{(D_j,U_j): j\in {\mathcal N}} \pr(\cdot), \label{eq:sum}
\end{align}
where
\[
\pr(\cdot)=\pr(\cap_{i=1}^k({\mathcal E}_{D_i}\cap {\mathcal E}_{U_i}) \mid M, \{D_P[s_{y_j}]=D_j:\, j<i\}, \{U_P[s_{y_j}]=U_j:\, j<i\}).
\]
{\em Explanation of~\eqn{eq:sum} and the notation used therein.\ } The first sum is over all integers $h,\ell,a$ that denote the sizes of $\Sigma_1$, $\Sigma_2$ and the total number of vertices in $W(M)$. The second sum is over all possible choices of $(\Sigma_1,\Sigma_2,R)$ that is consistent with $(h,\ell,a)$, where $R$ denotes the set of red vertices in $W(M)$. The third sum is over all possible weights on the arcs of $W(M)$, which indicates the lengths of the directed paths in $M$ between two vertices of $W(M)$. The next sum is over all possible choices of $M$ that are consistent with $(\Sigma_1,\Sigma_2,R,\wvec)$.  The next sum is over all choices of $k$ well-separated {\trilink}s.  Given a set of ${\trilink}$s ${\mathcal T}$, let ${\mathcal N}={\mathcal N}(\wvec)$ be the set of $j\in [k]$ such that $w^{(j)}<4Kc_n$ where $w^{(j)}$ is for $y_j$ where $(x_i,y_i,z_i)\in {\mathcal T}$. Next, we sum over all choices of $s_{y_j}$ for $j\in {\mathcal N}(\wvec)$. The final sum is over all possible subsets $D_j, U_j, j\in {\mathcal N}$ that are pairwise disjoint, except that $D_j\cap U_j=y_j$ for all $j$.
Recalling that $N(\wvec)=|{\mathcal N}(\wvec)|=|\{i: w^{(i)}<4Kc_n\}|$.
By Claim~\ref{claim:largeDownset}(a), for each $j\in {\mathcal N}(\wvec)$ either $|D_j|\ge 4Kc_n\alpha$ or $|U_j|\ge 4Kc_n\alpha$. Thus,
by Claim~\ref{claim:largeDownset}(b,c), 
\begin{align*}
\sum_{(D_j,U_j): j\in{\mathcal N}} &\pr(\cap_{i=1}^k({\mathcal E}_{D_i}\cap {\mathcal E}_{U_i}) \mid M, \{D_P[s_{x_j}]=D_j, U_P[s_{x_j}]=U_j:\, j<i\})\\
&\le \Big(2n^{-1}\exp(-Kc_n/2)\Big)^{N(\wvec)}.
\end{align*}
The number of choices for $(s_j)_{j\in {\mathcal N}}$ is at most $n^{N(\wvec)}$, and the number of choices for ${\mathcal T}$ is at most $2^a$ as shown earlier. Thus, by Claim~\ref{claim:M},
\begin{align}
\pr({\mathcal B})&\le \sum_{h,\ell,a} \sum_{(\Sigma_1,\Sigma_2,R)} \sum_{\wvec} \left(\prod_{j=1}^{a+1} \frac{w_j}{n} c_n^{w_j}/w_j!\right) n^{N(\wvec)} 2^a \Big(2n^{-1}\exp(-Kc_n/2)\Big)^{N(\wvec)}\nonumber\\
&=  \sum_{h,\ell,a} \sum_{(\Sigma_1,\Sigma_2,R)} \sum_{N}\sum_{\wvec: N(\wvec)=N} \left(\prod_{j=1}^{a+1} \frac{w_j}{n} c_n^{w_j}/w_j!\right) n^{N}  2^a \Big(2n^{-1}\exp(-Kc_n/2)\Big)^{N}\nonumber\\
&\le \sum_{h,\ell,a} \sum_{(\Sigma_1,\Sigma_2,R)} \sum_{N} 2^{a+1} \left(n^{-1} c_ne^{c_n}\right)^{a+1-(k-N)/2} \left(n^{-1} e^{-Kc_n}\right)^{(k-N)/2}\nonumber\\
& \hspace{8cm} \times  n^{N}  2^a \Big(2n^{-1}\exp(-Kc_n/2)\Big)^{N}.\label{eq:sum2}
\end{align}
{\em Explanation of~\eqn{eq:sum2}.\ } We call $j\in [a+1]$ heavy if $w_j\ge 2Kc_n$ and light otherwise. Since $N(\wvec)=N$, there are $k-N$ $y_i$s in a \trilink\ in ${\mathcal T}$ such that $w^{(i)}\ge 4Kc_n$. Since each of these $y_i$s has degree two in $W(M)$, at least one of the directed path incident to $y_i$ in $M$ has length at least $2Kc_n$, i.e.\ the corresponding $j$ is heavy. Since each directed path joins two vertices in $M$, there must be at least $(k-N)/2$ heavy $j\in [a+1]$. 
Let ${\mathcal A}$ be the set of all possible assignments of ``heavy'' and ``light'' to $j\in [a+1]$ such that there are at least $(k-N)/2$ heavy elements. Obviously, $|{\mathcal A}|\le 2^{a+1}$. 
Then,
\[
\sum_{\wvec: N(\wvec)=N} \left(\prod_{j=1}^{a+1} \frac{w_j}{n} c_n^{w_j}/w_j!\right) =\sum_{A\in {\mathcal A}}\sum_{\wvec\vDash A: N(\wvec)=N} \left(\prod_{j=1}^{a+1} \frac{w_j}{n} c_n^{w_j}/w_j!\right), 
\]
where the second sum above on the right is over all $\wvec$ consistent with a given assignment $A$, and the condition that $N(\wvec)=N$. By Corollary~\ref{cor:tools}(b), for every $A$,
\[
\sum_{\wvec\vDash A: N(\wvec)=N} \left(\prod_{j=1}^{a+1} \frac{w_j}{n} c_n^{w_j}/w_j!\right) \le  \left(n^{-1} c_ne^{c_n}\right)^{a+1-(k-N)/2} \left(n^{-1} e^{-Kc_n}\right)^{(k-N)/2} , 
\]
which completes the proof for~\eqn{eq:sum2}.

Notice, by recalling that $k=\lceil a/21-1\rceil\ge a/42$, that
\begin{align*}
2^{a+1} &\left(n^{-1} c_ne^{c_n}\right)^{a+1-(k-N)/2} \left(n^{-1} e^{-Kc_n}\right)^{(k-N)/2} n^{N}  2^a \Big(2n^{-1}\exp(-Kc_n/2)\Big)^{N}\\
&\le \left(\frac{8c_ne^{c_n}}{n}\right)^{a+1} \exp\left(-\frac{Kc_n}{2} \frac{k+N}{2}\right)\\
&\le \left(\frac{8c_ne^{c_n}}{n}\right)^{a+1} \exp\left(-\frac{Kc_n}{4\cdot 42} \cdot a\right), 
\end{align*}
and that the number of possible values of $N$ is at most $a$.
Thus,
\begin{align*}
\pr({\mathcal B})&\le \sum_{h,\ell,a} \sum_{(\Sigma_1,\Sigma_2,R)} a \left(\frac{8c_ne^{c_n}}{n}\right)^{a+1} \exp\left(-\frac{Kc_n}{168} \cdot a\right) \\
&\le \sum_{h,\ell,a} \binom{n}{a} \binom{a}{h,\ell} a^3 \left(\frac{8c_ne^{c_n}}{n}\right)^{a+1} \exp\left(-\frac{Kc_n}{168} \cdot a\right) \\
&\le \sum_{h,\ell,a} \frac{8c_ne^{c_n}}{n} \frac{1}{h!\ell!} \left(\frac{8c_ne^{c_n}}{e^{Kc_n/168}}\right)^a =o(1),
\end{align*}
by choosing sufficiently large $K$ so that $8c_n e^{c_n-Kc_n/168}<1/2$. \qed \smallskip

\noindent{\em Proof of Claim~\ref{claim:smallM}.\ }  Fix $B,h,\ell$. It suffices to show that a.a.s.\ there are no connected bicyclic subgraph $M$ in $G_D[S]$ such that $W(M)$ has exactly $B$ vertices, $h$ vertices in $\overline{M}$ and $\ell$ vertices in $\underline{M}$. The number of ways to fix the set of vertices in $M$ and the vertices in  $\overline{M}$ and  $\underline{M}$ respectively is at most $O(n^B)$. There are $O(1)$ ways to fix $W(M)$ on the set of chosen vertices. Finally, to form $M$, we consider all possible directed paths $P_{xy}$ where $(x,y)$ is an arc in $W(M)$. Hence, the probability that such a  connected bicyclic subgraph $M$ exists in $G_D[S]$  is at most
\[
O(n^B) \sum_{P_{1},\ldots, P_{B+1}} \pr(\cap_{j=1}^{B+1} {\mathcal E}_{P_j} ),
\]
where the summation is over all possible internally disjoint directed paths $P_{xy}$; noting that there are exactly $B+1$ directed paths since $W(M)$ is bicyclic. By Corollary~\ref{lem:tools}, this probability is at most
\[
O(n^B) \left(c_n e^{c_n}/n\right)^{B+1}=o(1),\quad \text{as $c_n=o(\log n)$}.\qed
\]

\noindent{\em Proof of Claim~\ref{claim:count}.\ }  There are $a$ vertices, and every vertex has at most four neighbours. Specifying the set of neighbours for each vertex specifies a graph in ${\mathcal F}$. This immediately implies the upper bound $a^{4a}$. \qed
\smallskip

\noindent{\em Proof of Claim~\ref{claim:M}.\ } By definition of ${\mathcal M}$,
\[
\sum_{M\in {\mathcal M}} \pr({\mathcal E}_M) = \sum_{P_1,\ldots, P_t} \pr(\cap_{j=1}^t {\mathcal E}_{P_j}),
\]
where the summation is over all possible internally disjoint directed paths where $P_{j}$ joins $x$ and $y$ if $(x,y)=e_j$ for every $j\in[t] $. By Lemma~\ref{lem:tools},
\[
\sum_{P_1,\ldots, P_t} \pr(\cap_{j=1}^t {\mathcal E}_{P_j}) \le \sum_{P_1,\ldots, P_t} \prod_{j=1}^t \pr({\mathcal E}_{P_j}) \le \prod_{j=1}^{t} \frac{w_j}{n} \frac{c_n^{w_j}}{w_j!} \qed
\]

\noindent{\em Proof of Claim~\ref{claim:largeDownset}.\ } Suppose that $w^{(i)}<4Kc_n$. Since $s_{y_i}\in S$, either down-set or the up-set of $s_{y_i}$ has size is at least $5Kc_n\alpha$ by~\eqn{def:S}. Since $|D_P[s_{y_i}]\cap M| \le w^{(i)}+1$ and $|U_P[s_{y_i}]\cap M| =1$, it follows immediately that  either $|D_P[s_{y_i}]\setminus M|$ has size at least $5Kc_n\alpha-4Kc_n\ge 4Kc_n\alpha$, or $|U_P[s_{y_i}]\setminus M|$ has size at least $5Kc_n\alpha-1\ge 4Kc_n\alpha$.

It only remains to prove part (b), as the proof for part (c) is analogous and slightly simpler. 
Let $U$ be the set of elements in $M$ that are in the (closed) down-set of $y_i$. As $K$ is some arbitrary large constant, we may assume that it is chosen so that $4Kc_n$ is an integer. By assumption $|U|\le 4Kc_n-1$. Given $M, \{D_P[s_{y_j}], U_P[s_{y_j}]:\, j<i\}$, let $\E_1$ be the event that $|D_P[U]|\ge (4Kc_n-1)\alpha$, $\E_2$ be the event that $|D_P[s_{y_i}]\setminus U|\ge \alpha$, and $\E_3$ be the event that there is a directed path from $y_i$ to $s_{y_i}$.
Finally, let $\E$ denote the event that $D_P[s_{y_i}]$ is
disjoint from $\{D_P[s_{y_j}], U_P[s_{y_j}]:\, j<i\}$, and $|D_P[s_{y_i}]\setminus M|\ge 4Kc_n\alpha$.
Note that $\sum_{D\in{\mathcal D}_i}\pr({\mathcal E}_D \mid M, \{D_P[s_{y_j}], U_P[s_{y_j}]:\, j<i\})$ is exactly $\pr(\E)$. Moreover, 
$\cup_{D\in {\mathcal D}_i}{\mathcal E}_D$ implies $(\E_1\cap\E_3)\cup (\E_2\cap \E_3)$.  Part (b) follows from the following bounds by choosing sufficiently large $K$:
\begin{align}
\pr(\E_1\cap\E_3) & =  O(n^{-1} c_n^2 e^{c_n}) \exp(-Kc_n),  \label{eq:E1}  \\
\pr(\E_2\cap\E_3) &=O\left(\frac{c_n}{n} \exp(-Kc_n)\right).  \label{eq:E2}
\end{align}

{\em Proof of~\eqn{eq:E1}.\ } It is obvious that $\E_1$ and $\E_3$ are independent and thus $\pr(\E_1\cap\E_3)=\pr(\E_1)\pr(\E_3)$. By Corollary~\ref{cor:tools} (b), $\pr(\E_3)\le c_n e^{c_n}n^{-1}$. The probability $\E_1$ is obviously maximized when $U$ is the largest $|U|$ elements in $[n]$ in the natural order. Moreover, since $|U|\le 4Kc_n-1$ and $|D_P[U]|\ge (4Kc_n-1)\alpha$, at least one of the element in $U$ must have a down-set of size at least $\alpha$. Hence, 
$$\pr(\E_1)\le |U| \cdot\pr(\text{element $n$ has a down-set of size $\alpha$})\le 4Kc_n\exp(-Kc_n)$$
by Corollary~\ref{cor:tools} (a). 

{\em Proof of~\eqn{eq:E2}.\ } Let $P'$ be the sub-poset induced by the elements in $[n]\setminus M$. Then,
\[
\pr(\E_2\cap\E_3) = \sum_{D} \pr(D_{P'}[s_{y_i}]=D) \pr(\text{there is a directed arc joining $y_i$ and $D\cup \{s_{y_i}\}$}),
\]
where the summation is over all $D$ as a subset of $[n]$ disjoint from $\{D_P[s_{y_j}], U_P[s_{y_j}]:\, j<i\}$ and from $M$ where $|D|\ge \alpha-2$ (the subtraction of 2 is for $y_i$ and $s_{y_i}$).  Thus, by~\eqn{eq:downset},
\begin{align*}
\pr(\E_2\cap\E_3) &=\sum_{s\ge \alpha-2}  \sum_{D: |D|=s} \pr(D_{P'}[s_{y_i}]=D) (1-q^{s+1}) \\
&\le \sum_{\alpha-2\le s\le n^{1/3}}  (1-q^{s+1}) q^{n-s} \prod_{j=0}^{s-1}(1-q^{j+1})  \prod_{i=1}^{s} \frac{1-q^{n-s+i}}{1-q^i} + \pr(|D_{P'}[s_{y_i}]|>n^{1/3}).
\end{align*} 
By Lemma~\ref{lem:tools}(a) and the assumption that $c_n=o(\log n)$, $\pr(|D_{P'}[s_{y_i}]|>n^{1/3})\le n^{-1} e^{-Kc_n}$. Now consider $s$ such that $\alpha-2\le s\le n^{1/3}$, and we obtain 
\begin{align*}
(1-q^{s+1})q^{n-s} &\sim \frac{c_ns}{n} e^{-c}  
\end{align*}
and
\begin{align*}
 \prod_{j=0}^{s-1}(1-q^{j+1})   \prod_{i=1}^{s} \frac{1-q^{n-s+i}}{1-q^i} &=  \prod_{i=1}^{s} (1-q^{n-s+i}) \\
 &= \prod_{i=1}^{s} (1-e^{-c_n}) \left(1+O\left(\frac{c_n}{n}(s-i)\right)\right) \sim (1-e^{-c_n})^s.
\end{align*}
It follows that
\begin{align*}
\pr(\E_2\cap\E_3) &
\le \sum_{\alpha-2\le s\le n^{1/3}}  (1+o(1)) \frac{c_ns}{n} e^{-c}   (1-e^{-c_n})^s +n^{-1} e^{-Kc_n}\\
& = O\left(\frac{c_ne^{-c_n}}{n}  (1-e^{-c_n})^{Kc_ne^{c_n}}\right) +n^{-1} e^{-Kc_n} = O\left(\frac{c_n}{n} \exp(-Kc_n)\right).\qed
\end{align*}






\end{document}